\documentclass[11pt, a4paper]{article}
\usepackage{amsmath,amsthm,amscd,amsfonts,amssymb,enumitem}
\usepackage{latexsym}
\usepackage{multirow}
\usepackage{xcolor}
\usepackage{lipsum} % for filler text
\usepackage{comment}

\usepackage[colorlinks=true, linkcolor=blue, anchorcolor=blue, citecolor=blue, filecolor=blue, urlcolor=blue]{hyperref}

\newtheorem{thm}{Theorem}[section]
\newtheorem{lemma}[thm]{Lemma}

\newtheorem{prop}[thm]{Proposition}
\theoremstyle{definition}
\newtheorem{defin}[thm]{Definition}

\newtheorem{rem}[thm]{Remark}

\newcommand{\qedwhite}{\hfill \ensuremath{\Box}}
\renewenvironment{proof}{{\raggedright \bfseries Proof.}}{\qedwhite}

\numberwithin{equation}{section}
\newcommand\blfootnote[1]{%
	\begingroup
	\renewcommand\thefootnote{}\footnote{#1}%
	\addtocounter{footnote}{-1}%
	\endgroup
}

\def\CC{\mathbb C}
\def\NN{\mathbb N}
\def\ZZ{\mathbb Z}
\def\RR{\mathbb R}

\def\Chi{\mbox{\Large$\chi$}}
\def\sinc{\text{sinc}}
\def\esssup{\mathop{\textup{ess}\sup}}

\newcommand{\vc}[1]{\overrightarrow{#1}}

\title{Constructive Approximation in Mixed norm Spaces}

\author{
	Priyanka Majethiya\textsuperscript{1}\thanks{Email: priyankamajethiya2000@gmail.com}, Shivam Bajpeyi\textsuperscript{1}\thanks{Corresponding Author. Email: shivambajpai1010@gmail.com, shivambajpeyi@amhd.svnit.ac.in}, Dhiraj Patel\textsuperscript{2}\thanks{Email: patel@cs.rwth-aachen.de}\\
	{\small \textsuperscript{1}Department of Mathematics, Sardar Vallabhbhai National Institute of Technology Surat,}\\ {\small Gujarat-395007, India}\\
	{\small \textsuperscript{2}Computational Network Science, RWTH Aachen University, Aachen, Germany}
}

\date{}

\begin{document}
	
	\maketitle

\begin{abstract}
    The concept of mixed norm spaces has emerged as a significant interest in fields such as harmonic analysis. In addition, the problem of function approximation through sampling series has been particularly noteworthy in approximation theory. In this paper, we study the approximation problem in diverse mixed norm function spaces. We utilize the family of Kantorovich-type sampling operators as approximator for the functions in mixed norm Lebesgue space $L^{\overrightarrow{P}}(\mathbb{R}^n),$ with $ \overrightarrow{P}=(p_1,p_2,\dots,p_n),$ and mixed norm Orlicz space $L^{\overrightarrow{\Phi}}(\mathbb{R}^n),$ with $\overrightarrow{\Phi}=\left(\phi_{1},\phi_{2},\dots, \phi_{n} \right),$ where each $\phi_{i}$ is an Orlicz function (defined in Section \ref{2}). The Orlicz spaces are a generalized family that encompasses	many significant function spaces. We establish the boundedness of the family of generalized and Kantorovich-type sampling operators within the framework of these mixed norm spaces. Further, we study the approximation properties of Kantorovich-type sampling operators in both mixed norm Lebesgue and Orlicz spaces. Lastly, we provide several examples of appropriate kernels that demonstrate the applicability of the proposed theory.
    
    % involved in the discussed approximation procedure.
    \paragraph{Keywords:}  Approximation of functions; Kantorovich-type sampling series; Mixed norm Lebesgue space; Mixed norm Orlicz space.
    \blfootnote{2020 \textit{Mathematics Subject Classification:} 94A20, 26D15, 46B09, 46E30, 47B34.}
\end{abstract}

\section{Introduction}

In broad terms, the aim of function approximation is to identify a simpler function that closely represents a complicated function. Sampling and reconstruction of functions have been playing a pivotal role in approximation theory, serving as fundamental tools for accurately representing functions through discrete sample values \cite{branch}. The pioneer study of sampling and reconstruction problem for bandlimited functions was carried out in \cite{ but,survey,noise}.
\par 
A fundamental challenge in signal processing is to recover a signal $f$ from their discrete sample values. Whittaker-Kotelnikov-Shannon (WKS) \cite{survey,past, vgood} provide a reconstruction formula to recover bandlimited signal from their discrete sample measurements.
In particular, a bandlimited signal $f:\RR \rightarrow \CC$ with bandwidth $2\pi w$, for some $w>0$, i.e., the Fourier transform of $f$ is supported on $[-\pi w,\pi w]$, can be entirely reconstructed from their samples values at $(\frac{k}{w})_{k \in \ZZ}$ by using the formula
\begin{equation}\label{shannon}
    f(x)=\sum_{k\in \ZZ} f\left(\frac{k}{w} \right)\sinc(xw-k),\ x \in \RR,
\end{equation}
where $\sinc(x) :=\frac{\sin \pi x} {\pi x}$ for $x\neq 0$ and $\sinc(0)=1$.
The WKS sampling theorem is the classical sampling theorem originally formulated by Shannon, serving as a fundamental concept in communication theory, information theory, and signal processing \cite{survey,  branch,jerry}. 
Moreover, the application is not limited to communication theory but also have remarkable advancements in areas such as optics, time-varying systems, boundary value problems, approximation and interpolation theory, and Fourier transform \cite{branch,jerry,noise}.
Some notable developments related to the WKS sampling theorem are detailed in \cite{delta,survey, butzer,  vgood,  jerry}.

\paragraph{Generalized Sampling Series.} 
% \subsection{Generalized Sampling Series}
The WKS sampling theorem was generalized for the signals in the shift-invariant space, image of idempotent integral operator and so on \cite{shift}. The key concept is to recover the signal uniquely and stably from their sample values without losing any information. In (\ref{shannon}), the sinc kernel is utilized, due to its computational limitations and the fact that $\sinc\,t$ is not absolutely integrable over \(\mathbb{R}\) \cite{vgood}. To address this, Butzer  \cite{butzer1, past} introduced a generalized sampling series using the discrete sample values of the signal $f$ and kernel $\Chi$.
The generalized sampling series for a signal $f$ is given by 
\begin{equation}\label{gen1}
    S_w(f)(x) :=\sum_{k\in \ZZ} f\left(\frac{k}{w} \right)\Chi(wx-k),\quad x \in \RR, \quad w>0,
\end{equation}
where $\Chi$ denotes the kernel function satisfying certain assumptions \cite{past}. 
The pointwise and uniform convergence for (\ref{gen1}) was initially discussed in the foundational work of Butzer and Stens \cite{butzer1, past}.
Bardaro et al. \cite{delta} extended the classical Shannon sampling series (\ref{shannon}) by introducing an $L^p$ averaged modulus of smoothness to characterize intricate signals. The reconstruction of $p$-integrable functions by generalized sampling series was further investigated in \cite{vgood}. 
Additionally, Butzer et al. \cite{butzer1} studied the approximation properties of generalized sampling series for both continuous and discontinuous functions. In a subsequent work, the same author developed convergence results for the multivariate generalized sampling series in the space of continuous functions \cite{multigen}. Further details on the properties and applications of these operator can be found in  \cite{delta, but,multigen,butzer1,  vgood}.
    
\paragraph{Kantorovich-type Sampling Series} Later,  Bardaro et.al. \cite{bardaro10} introduced the Kantorovich-type sampling series. Instead of taking sample values, these series compute the average value of the function over each interval, namely $w \int_{\frac{k}{w}}^{\frac{k+1}{w}}f(u)du$, for $w>0$. For  a locally integrable function $f:\RR \rightarrow \RR$, the Kantorovich-type sampling series is defined as
\begin{equation}\label{kantt}
    K_{w}(f)({x}) :=\sum_{{k}\in \ZZ} \Chi(w{x}-{k})\hspace{3pt} w \int_{\frac{k}{w}}^{\frac{k+1}{w}}f(u)du, \quad  x \in \RR, \quad w>0 .
\end{equation} 
The approximation of continuous functions by Kantorovich-type sampling series have been established in \cite{bardaro10, multi}, while the effectiveness of the series defined in (\ref{kantt}) have also been demonstrated for reconstructing signals that may not be continuous \cite{orlova}.
Although the averaging techniques improve the regularity of these series, achieving convergence at the discontinuity points of $f$ remains a challenge due to technical limitations. This issue is closely related to the application of Kantorovich-type sampling series in image reconstruction and enhancement.

Recently, Costarelli et al. \cite{disc} studied the approximation properties of (\ref{kantt}) for discontinuous functions across various fields, including biomedical and engineering applications \cite{cagini, civil}.
Cantarini et al. \cite{diff} made a significant contribution to the study of both differentiable and not differentiable functions for (\ref{kantt}). In \cite{morder}, the convergence behavior of $m$- order Kantorovich-type sampling operators is discussed. The generalized Kantorovich-type sampling operators is introduced in \cite{orlova},
where local averages are determined through the convolution
with an arbitrary function $\Chi \in L^1(\RR)$.  Bardaro et al. \cite{bardaro10} pioneered the study of the convergence of the operators $(K_{w})_{w>0}$ in the framework of Orlicz spaces. This line of research have been further extended to the multivariate form of Kantorovich-type sampling operators in \cite{multi}. The development of the theory in a multivariate framework is important for applications, particularly in signal theory and image processing.
Several recent developments concerning the family (\ref{kantt}) are discussed in \cite{bardaro10, diff, disc, multi, orlova}.

This paper aims to investigate the approximation properties of a family of sampling operators within the framework of a mixed norm structure. We study its approximation capabilities in a broad class of functions spaces, which may not be continuous, such as mixed norm Lebesgue spaces and mixed norm Orlicz spaces.  First, we examine the boundedness of generalized sampling operators in mixed norm Lebesgue space. We then establish the boundedness and convergence of Kantorovich-type sampling operators within mixed norm Lebesgue spaces. Further, we extend these results in the setting of mixed norm Orlicz spaces.

% \subsection{Orlicz space}
The theory of \emph{Orlicz spaces}, which provides a natural generalization of Lebesgue spaces, was established by W. Orlicz in \cite{or}. It is noteworthy that the Orlicz spaces, denoted as $L^{\phi},$ consist of many significant function spaces for different choice of $\phi-$functions, namely classical Lebesgue spaces, Exponential spaces and Logarithmic spaces (also known as Zygmund spaces).
The seminal work on Orlicz spaces can be found in \cite{orlicz,1}, and since then, various researchers have studied their properties and potential applications in diverse areas. The rigorous framework of Orlicz spaces and their properties was introduced by Rao and Ren \cite{Raobook}, along with related topics such as compact sets, the product of Orlicz spaces, topological structures, and more.  Moreover, \cite{orliczapp} provides a detailed discussion of fundamental applications of Orlicz spaces in areas such as Fourier analysis, stochastic analysis, and nonlinear PDE.  Function spaces derived from Orlicz spaces have proven highly effective in the study of functional analysis and operator theory \cite{sob, kita}. Submultiplicative functions play an important role in interpolation theory and operator theory within Orlicz spaces, certain results are presented in \cite{2}. Hence, Orlicz spaces offer a unified framework for studying Kantorovich-type sampling operators in different function spaces. In particular, mixed norm function spaces are considered due to their flexible structure and broad applicability.

% \subsection{Mixed norm spaces} 
In \emph{mixed norm spaces}, the norm of measurable multivariable functions is determined by the iterative  application of different norm functions. The foundation for this concept is found in the pioneering article by Benedek and Panzone \cite{dct}. For instance, unlike classical Lebesgue spaces, in mixed norm Lebesgue spaces, the exponent for integrability can differ for each variable. This flexibility
makes it more useful for areas where different variables may have different properties.  
These spaces are particularly relevant for time-varying signals \cite{sun} and playing a pivotal role in examining the solutions to partial differential equations that involve both time and spatial variables, such as the heat and wave equations \cite{heat}. 
In \cite{dct,burenkov,grey, vitali, pq}, mixed norm spaces are examined from an operator-theoretic perspective. The necessary and sufficient condition for the embedding of mixed norm Lebesgue spaces and mixed norm Orlicz spaces are discueed in \cite{pq}. Maligranda’s work \cite{3} focused on the Calderón–Lozanovskiĭ construction, particularly its application to mixed norm function spaces generated from Banach ideal spaces. His research included a rigorous analysis of the Calderón product associated with these function spaces.

To the best of our knowledge, these operators have not been studied in mixed norm Lebesgue spaces and mixed norm Orlicz spaces. Given the enduring interest in Kantorovich-type sampling operators in approximation theory, examining their behavior in mixed norm settings is of particular significance.

\paragraph{Contributions.} The key features of the article are as follows.
\begin{itemize}
    \item   In \cite{delta, bardaro10, multi}, the boundedness and convergence of generalized and Kantorovich-type sampling series have been studied in Lebesgue space and Orlicz space. Here, we investigate these approximation properties for aforementioned family of sampling operators within mixed norm structures.  In particular, the mixed norm Orlicz spaces have received limited attention in the literature, this study also contributes to the theoretical advancement of these spaces.
   
    \item This paper provides a broader framework for studying various classes of function spaces, due to the general structure of Orlicz spaces. On the way of establishing the convergence of Kantorovich-type sampling series in mixed norm Orlicz space $ L^{\overrightarrow{\Phi}}(\mathbb{R}^n)$, we prove a fundamental result asserting that, the space of compactly supported continuous functions $C_{c}(\RR^{n})$ is dense in mixed norm Orlicz space $ L^{\overrightarrow{\Phi}}(\mathbb{R}^n)$ (see Lemma \ref{ce}).
\end{itemize}

\paragraph{Outline.}
The paper is structured as follows. In Section \ref{2}, we begin by recalling some fundamental definitions and results concerning mixed norm Lebesgue spaces and Orlicz spaces, which are requisite for the proposed study. In Section \ref{3}, we prove the boundedness of generalized sampling operators in $L^{\overrightarrow{P}}(\RR^n)$, followed by we demonstrate the boundedness and convergence of Kantorovich-type sampling operators in  $L^{\overrightarrow{P}}(\RR^n)$ in Section \ref{4}. 
In Section \ref{5}, we extend these results to the setting of mixed norm Orlicz space $L^{\overrightarrow{\Phi}}(\RR^n)$. In order to prove these results, we establish a density result in mixed norm Orlicz space.
In Section \ref{6}, we present several examples of kernel satisfying the assumptions (given in Section \ref{2}) of presented theory. 
 
\section{Preliminaries} \label{2} 
In this section, we define some basic definitions and discuss preliminary results on mixed norm Lebesgue space and mixed norm Orlicz space.

\subsection{Notations}
The notations $\NN^n, \ZZ ^n $ and $\RR ^n$ represent the set of all ordered $n$-tuples of natural numbers, integers and real numbers respectively. Let $C(\mathbb{R}^n)$ denotes the set of all bounded and uniformly continuous functions, it also forms a normed vector space with the usual sup norm $\displaystyle \|f\|_{\infty}:= \esssup_{\textbf{t} \in \RR^n} |f(\textbf{t})|.$  For $\textbf{x}=(x_1,\dots,x_n)$  and  $\textbf{y}=(y_1,\dots,y_n)$, we write $(\textbf{x}+\textbf{y})=(x_1+y_1,\dots,x_n+y_n)$, and
$\zeta\textbf{x}=(\zeta x_1,\dots,\zeta x_n)$, for any scalar $\zeta \in \RR$.
% Here $\textbf{x} = \textbf{y}$ denotes that $x_i = y_i$ for each $i=1,2,\dots,n$.
The space $L^{p}(\mathbb{R}^n), \hspace{4pt}\text{for}\ 1 \leq p < \infty$ contains of all $p$-integrable functions on $\mathbb{R}^n$ with the standard Lebesgue $p-$norm. The space $L^{\infty}(\mathbb{R}^n)$ refers to the collection of all measurable functions that are bounded and equipped with the sup norm $\|\cdot\|_{\infty}$. Furthermore, $M(\RR^n)$ is the space of all bounded and measurable real valued functions.
    
\subsection{Mixed norm Lebesgue space}\label{subsection:mixLebesgue}
Let $(\Omega_i,S_i,\mu_i)$ be the $\sigma$-finite measure spaces, for $1 \leq i \leq n $, and $(\Omega,S,\mu)$ represent corresponding product measure space, i.e., $\Omega=\prod\limits_{i=1}^{n}{\Omega_i}.$ 
% In this section, we revisit the theory of mixed norm Lebesgue spaces.
The mixed norm Lebesgue space is defined as follows.
\begin{defin}\label{def:mixed_lebesgue}(Mixed norm Lebesgue space)
    For $1 \leq p_i \leq \infty$ and $\overrightarrow P=(p_1,\dots,p_n)$, the mixed norm Lebesgue space $L^{\vc{P}}(\Omega)$ is defined as
    \begin{flalign*}
        L^{\vc{P}}(\Omega)=\left\lbrace f :\Omega \rightarrow \RR \text{ measurable } : \|f\|_{\vc{P}} <\infty \right\rbrace,
    \end{flalign*}
    where
    \begin{equation*} \label{2.1}
		\|f\|_{\vc{P}} = \bigg( \int_{\Omega_n}\dots\bigg( \int_{\Omega_2}\bigg( \int_{\Omega_1} |f(x_1,\dots,x_n)|^{p_1}d\mu_1\bigg) ^{\frac{p_2}{p_1}}d\mu_2\bigg) ^{\frac{p_3}{p_2}}\dots d\mu_n\bigg) ^{\frac{1}{p_n}}.
    \end{equation*}
\end{defin}

The mixed norm Lebesgue space can equivalently be viewed within the framework of vector-valued Lebesgue spaces, represented as follows: \[ L^{\vc{P}}(\Omega) = L^{p_n}_{x_n}\left( \Omega_n, L^{(p_1, p_2, \dots, p_{n-1})}_{x_1, x_2, \dots, x_{n-1}} \left( \prod_{k=1}^{n-1} \Omega_k\right) \right).\] In the subsequent analysis, we only focus on the tuples $\vc{P}=(p_1,p_2,\dots,p_n)$ with non-decreasing components, i.e., $p_i\leq p_{i+1}$ for $1\leq i\leq n-1$. Although, this is a strong assumption and does not encompass all possible mixed norm Lebesgue space $L^{\vc{P}}(\Omega)$, it still generalizes the classical Lebesgue space $(L^{p}(\Omega),\|\cdot\|_p)$. In particular, when $\vc{P}=(p,p,\dots,p)\in [1,\infty)^{n}$ is constant vector, the mixed norm space $\big(L^{\vc{P}}(\Omega), \|\cdot\|_{\vc{P}}\big)$ is equivalent to the standard Lebesgue space $(L^{p}(\Omega),\|\cdot\|_p)$. This monotonicity assumption on the exponents is not only analytically convenient but also appears in the study of sampling and reconstruction within reproducing kernel subspaces of mixed norm Lebesgue space \cite{anujmixed}.
%    In the mixed norm  Lebesgue space,  for $f\in L^{\overrightarrow P}(Y) $ and $ g\in L^{\overrightarrow Q}(Y)$, the \textit{H\"{o}lder's inequality} \cite{dct} states that\\
% \begin{equation*}\label{holder}
%     \|f g\|_1 \leq \|f\|_{\overrightarrow{P}}\|g\|_{\overrightarrow{Q}},
% \end{equation*}
% where $ {\overrightarrow P= (p_1,\dots,p_n)\in [1,\infty]^{n}}$ and ${\overrightarrow Q= (q_1,\dots,q_n)\in [1,\infty]^{n}}$, where each pair satisfies  $\frac{1}{p_k} +\frac{1}{q_k}=1$ for each $k=1,\dots,n$.

\begin{defin}
    The space $\ell^{\vc{P}}(\ZZ^n)$ is a collection of sequences on $\ZZ^n$ with the mixed norm defined as follows:
    \begin{flalign*}
        \|x\|_{\ell^{\vc{P}}}=\bigg( \sum_{k_n \in \ZZ}\cdots \bigg( \sum_{k_2 \in \ZZ}\bigg(  \sum_{k_1\in \ZZ}\left|x(k_1,k_2,\dots,k_n) \right|^{p_1}\bigg) ^{\frac{p_2}{p_1}}\bigg)^{\frac{p_3}{p_2}}\cdots\bigg)^{\frac{1}{p_n}}<\infty.
    \end{flalign*}
\end{defin}
For a positive real $w$, we define the space of weighted sequence space $\ell^{\overrightarrow{P}}_w(\ZZ^n)$ such that $$\|x\|_{\ell^{\overrightarrow{P}}_w} :=\bigg( \frac{1}{w}\sum_{k_n \in \ZZ}\cdots \bigg( \frac{1}{w}\sum_{k_2 \in \ZZ}\bigg(  \frac{1}{w}\sum_{k_1\in \ZZ}\left|x(k_1,k_2,\dots,k_n) \right|^{p_1}\bigg) ^{\frac{p_2}{p_1}}\bigg)^{\frac{p_3}{p_2}}\cdots\bigg)^{\frac{1}{p_n}}<\infty.$$
For detailed information on mixed norm Lebesgue spaces, we refer \cite{ivec, dct,burenkov,grey, 3, pq}.

\subsection{The Orlicz spaces}\label{s2.3}
A convex and lower continuous function $\phi:[0,\infty)\to [0,\infty]$ is called an Orlicz function if it satisfies the following conditions: 
\begin{enumerate}[label=(\roman*), left=0pt]\label{2.3}
    \item  $\lim\limits_{u\to \infty} \phi(u)=\infty$, and  $\phi(0)=0$,
			
    \item $\phi(u)>0$ for  every $u>0$, and is non-decreasing on $[0,\infty)$.
\end{enumerate}
It is important to note that \(\phi\) may take the value \(\infty\) and is not necessarily continuous. For instance, for a fixed \( b > 0 \),
\[	\phi(u) = 
\begin{cases} 
    0, & \text{if } u \leq b \\
    \infty, & \text{if } u > b,
\end{cases}\]
defines an Orlicz function. Orlicz functions generalize the power functions used in defining classical $L^p$ spaces and allow for more flexible control over growth and integrability conditions.
\par 
To construct a function space associated with an Orlicz function $\phi$, we introduce a fundamental building block: the modular functional associated with $\phi$. 
 Given a measurable space $(\Omega, \mu)$, 
the functional $I^{\phi}: M(\Omega) \rightarrow [0,\infty]$ is defined by 
\begin{equation*} 
    I^{\phi}(f):=\int_{\Omega} \phi(|f(x)|) d\mu .
\end{equation*}
This functional generalizes the expression $\int |f|^p\,d\mu$ from classical Lebesgue theory and is central to defining the Orlicz space.
The functional $I^{\phi}$ is called a modular functional and satisfies the following properties for every $f,\,g \in  M(\Omega) :$
\begin{enumerate}[label=(\roman*)]
    \item $I^{\phi}(f)=0$ if and only if $f=0,$
    \item $I^{\phi}(-f)=I^{\phi}(f),$
    \item $I^{\phi}(\gamma f +\beta g) \leq I^{\phi}(f)+I^{\phi}(g) , \hspace{3pt}for\hspace{3pt} \gamma,\beta \geq 0, \hspace{3pt} \gamma +\beta =1.$
\end{enumerate}
Moreover, for any $f \in  M(\Omega)$, the map $\alpha\mapsto I^{\phi}(\alpha f)$ is non-decreasing for $\alpha >0$.
\begin{defin} (Orlicz space)
    Let $\phi$ be an Orlicz function. Then the Orlicz space generated by $\phi$ is given as $$L^{\phi}(\Omega)= \left\lbrace  f \in M(\Omega) : I^{\phi}(\lambda f) < \infty \hspace{4pt} \text{for some } \lambda >0 \right\rbrace, $$ equipped with the norm $$\|f\|_{\phi}= \inf \left\lbrace \lambda >0 : \int_{\Omega} \phi\left(\frac{{|f(x)|}}{\lambda}\right)d\mu \leq 1 \right\rbrace.$$
    % Also, the subspace of $L^{\phi}(\Omega)$ is defined as $$E^{\phi}(\Omega)= \left\lbrace  f :\Omega \rightarrow \CC \hspace{5pt} \text{measurable} : I^{\phi}(\lambda f) < \infty \hspace{4pt} \text{for every } \lambda >0 \right\rbrace. $$
\end{defin}
Several notable examples of Orlicz functions, along with their generated Orlicz spaces, are outlined below.
\begin{enumerate} [label=(\roman*), left=0pt]
    \item Let \( p \in [1, \infty] \), and define $\phi(y)=y^p$ when $p \in [1, \infty).$ 
    For $p=\infty$ set \[\phi(y) = 
    \begin{cases}
        0, & \text{when }  y \leq 1, \\
        \infty, & \text{when } y > 1.
    \end{cases}\]
    This generates the classical Lebesgue spaces $L^p(\RR)$, $1\leq p \leq \infty.$
    
    \item $\phi(y)=(y-1)\chi_{[1,\infty)}(y)$, associated with the combined space $L^1(\RR^n)+L^{\infty}(\RR^n)$.
	
    \item $\phi(y)=y\log y$, which defines the logarithmic space $L\log L(\RR^n)$.
	
    \item $\phi(y)=\exp(y^a)-1$ for $a>0$, leading to Exponential spaces $\exp L^a$.
\end{enumerate}
	
%%%%%%%%%%%%%%%%%%%
\begin{defin}($\Delta_2-$condition)\label{delta2}
    The $\Delta_2-$condition for an Orlicz function $\phi$ is satisfied if there exists a constant  $K>0$,  such that $$\phi(2u)\leq K\phi(u), \quad \forall u>0.$$
\end{defin}
The Orlicz space associated to Orlicz function with $\Delta_2$-condition provides important regularity properties. In particular,
\begin{itemize}[itemsep=0pt]
    \item For every $f\in L^{\phi}(\Omega)$, $I^{\phi}(\lambda f)<\infty$ for every $\lambda>0$.
    
    \item Given a Banach space structure, modular convergence and norm convergence coincide.
    
    \item For $\phi, \psi\in \Delta_2$, where $\psi$ is Orlicz conjugate of $\phi$, the space $L^{\phi}(\Omega)$ is reflexive.
\end{itemize}
Examples of Orlicz functions that satisfy the $\Delta_{2}-$condition are  $\displaystyle \phi(u)=u(\log(e+u))$ and $\displaystyle \phi(u)=u^p,$ for $ 1 \leq p <\infty$. However, $\phi(u)=(e^u-u-1)$ does not belong to $\Delta_{2}$. 

\begin{defin}(Jensen's integral inequality) 
    The Jensen's integral inequality  is satisfied for every function \( f \in L^{\phi}(\Omega) \) provided the following inequality is holds 
    \[ \phi\left( \frac{1}{\mu(\Omega)}\int_\Omega f(x)\,d\mu\right) \leq \frac{1}{\mu(\Omega)}\displaystyle \int_\Omega \phi(f(x))\,d\mu.\]
\end{defin}

% \begin{defin} (Modular dense)
% A set $A \subset X$ is called \emph{modular-dense} in $X$ if for every function $f \in X$, there exists a sequence of functions $(f_k) \subset A$ and $\lambda > 0$ such that (\cite{mod})
% \[
% \lim_{k\to\infty}I^{\phi}\big(\lambda (f_k - f)\big) =0.
% \]
% \end{defin}
For detailed discussions on the definitions and results concerning Orlicz spaces, see \cite{or1, 2,orlicz, 1,  mode, mod, Raobook, orliczapp}.

\subsection{The Mixed norm Orlicz spaces}
% Assume that $(\Omega_i,\Sigma_i,\mu_i)$ are measure spaces  for $1\leq i \leq n$, where $\Omega_i$ is a set, $\Sigma_i$ is a $\sigma$-algebra of  $\Omega_i$ and $\mu_i$ is a $\sigma$-additive measure on $\Sigma_i$, with $\Omega=\displaystyle \prod_{i=1}^{n}\Omega_{i}$.

% We now define a functional to establish our main results. 
In Subsection \ref{s2.3}, we introduced classical Orlicz spaces based on a single Orlicz function. We now extend this framework to the setting of mixed norms.

Let $\vc{\Phi}=(\phi_1,\phi_2,\dots,\phi_n)$ be a finite sequence of Orlicz functions. The modular functional $I^{\vc{\Phi}}: M(\Omega) \to [0,\infty]$ associated to $\vc{\Phi}$ is defined by 
\begin{equation*} \label{2.2}
    I^{\overrightarrow{\Phi}}(f):=\int_{\Omega_n} \phi_{n}\cdots \left( \int_{\Omega_2}\phi_{2} \left( \int_{\Omega_1} \phi_1 \left( |f(x_1,\dots,x_n)| \right)d\mu_1\right)d\mu_{2}\cdots \right)d\mu_n.
\end{equation*}

\begin{defin}(Mixed norm Orlicz space)
    The mixed norm Orlicz space generated by $ \overrightarrow{\Phi}=(\phi_1,\phi_2,\dots,\phi_n)$ is defined as
    $$L^{\overrightarrow{\Phi}}(\Omega) :=\left\lbrace f \right. \in M(\Omega) : I^{\overrightarrow{\Phi}}(\lambda f) < \infty \hspace{3pt} \text{ for some }\hspace{2pt} \lambda >0 \left.\right\rbrace $$
    induced with the norm
    $$\|f\|_{\overrightarrow{\Phi}}= \inf \left\lbrace \lambda >0 : I^{\vc{\Phi}}(f/\lambda) \leq 1 \right\rbrace. $$	
\end{defin}
 
% It is important to mention that the mixed norm Orlicz space $L^{\overrightarrow{\Phi}}(\RR ^n)$ is a linear space and it contains a linear subspace
%  $$E^{\overrightarrow{\Phi}}(\RR^n) :=\left\lbrace f \right. \in M(\RR^n) : I^{\overrightarrow{\Phi}}(\lambda f) < \infty \hspace{3pt} \text{for every} \hspace{2pt}\lambda >0 \left.\right\rbrace. $$
% The subspace $E^{\overrightarrow{\Phi}}(\RR^n) $ is referred to as the set of all finite elements within $L^{\overrightarrow{\Phi}}(\RR^n)$.
 
% Also $E^{\overrightarrow{\Phi}}(\RR^n) = L^{\overrightarrow{\Phi}}(\RR^n)$ if and only if $\overrightarrow{\Phi} \in \Delta_{2}$.

The space $L^{\vc{\Phi}}(\Omega)$ generalizes the notion of mixed norm Lebesgue space $L^{\vc{P}}(\Omega)$, as introduced in Subsection \ref{subsection:mixLebesgue}, paritcularly under the assumption that the components of $\vc{P}=(p_1,p_2,\dots, p_n)$ are non-decreasing. Specifically, the Orlicz functions are chosen as $\phi_1(x)=x^{p_1}$ and $\phi_i(x)=x^\frac{p_i}{p_{i-1}}$, for $2 \leq i \leq n$, the corresponding mixed norm Orlciz space $L^{\vc{P}}(\Omega)$ coincides with mixed norm Lebesgue space $L^{\vc{P}}(\Omega)$.

We say that $\overrightarrow{\Phi}=(\phi_1,\dots,\phi_n)\in \Delta_{2}$, if each $\phi_i$ satisfies $\Delta_{2}$-condition. The notion of modular convergence in $L^{\overrightarrow{\Phi}}(\Omega)$ plays an important role to study the norm convergence. A sequence of functions $(f_k)_{k > 0}$ in $L^{\overrightarrow{\Phi}}(\Omega)$ is said to modularly convergent to a function $f \in L^{\vc{\Phi}}(\Omega)$ if $$\lim_{k\to\infty}I^{\vc{\Phi}}(\lambda(f_k-f))=0$$ for some $\lambda >0.$ If $\vc{\Phi}\in \Delta_2$ then for every $\lambda > 0$, $I^{\vc{\Phi}}(\lambda(f_k-f)) \to 0$ as $k\to \infty$. Hence, $\|f_k-f\|_{\vc{\Phi}}\to 0$ as $k\to \infty$. Moreover, the norm convergence and modular convergence are equivalent if and only if $\overrightarrow{\Phi}$ satisfies the $\Delta_2$-condition. 
% In this study, we consider a family of Orlicz function $\phi_i$, $1 \leq i \leq n$, each of which satisfies the $\Delta_{2}$ condiiton. 

% The mixed norm Orlicz space defined in this paper does not generalize the mixed norm Lebesgue space defined in Definition \ref{def:mixed_lebesgue}, however, it does generalize the space of mixed norm Lebesgue space $L^{\vc{P}}$ for $p_i\leq p_{i+1}$. 
% In the rest of the paper, we consider such mixed norm Lebesgue space $L^{\vc{P}}$ to study generalized sampling series and Kantcorovich-type sampling operator.}
%%%%%%%%%%%%%%%%%%%%%%%%%%%%%

The Shannon sampling series (\ref{shannon}) is limited by the assumption of band-limited signals \cite{vgood}. To extend this study to a more general setting, Butzer et al. \cite{multigen} introduced the generalized sampling series. 
In the subsequent section, we study the generalized sampling series within the framework of mixed norm Lebesgue space.
\section{Generalized sampling operator on mixed norm Lebesgue space}\label{3}
A function is said to be a kernel $\Chi: \RR^{n}\rightarrow \RR$ if it belongs to $L^{1}(\RR^n)$, is bounded in a neighborhood of the origin, and fulfills the following conditions:
\begin{enumerate}
    \item For any ${\textbf {u}}\in \RR^n$, it follows that 
    \begin{equation} \label{re1}
        \sum_{\textbf{k}\in \ZZ^n}\Chi(\textbf{u}-\textbf{k})=1 .
    \end{equation}

    \item  Absolute moment of order $\alpha$ is finite, i.e.,
    \begin{equation}\label{re2}
        m_{\alpha}(\Chi) := \sup_{\textbf{u} \in \RR^n}\sum_{\textbf{k}\in \ZZ^n}|\Chi(\textbf{u}-\textbf{k})| \hspace{2pt}\|\textbf{u}-\textbf{k}\|^{\alpha}_2<+\infty,
    \end{equation}
    where $\|\cdot\|_2$ denotes the euclidean norm.
\end{enumerate}

The reader can easily verify that the kernel $\Chi\in L^1(\RR^n)$ with \eqref{re1} and \eqref{re2} implies $$m_0(\Chi) \geq \max\{1,\|\Chi\|_1\}.$$

% In the following, we present some important definitions used in this section.
% examine the behavior of the generalized sampling series, utilizing the concept of a relatively separated set, which is defined below.
\begin{defin}(Relatively separated set)
    A countable set \(\Gamma \subset \mathbb{R}^n\) is defined to be \emph{relatively separated set} with gap \(\kappa > 0\) if it satisfies the following conditions:
    \begin{align*}
        A_\Gamma(\kappa) :=& \inf_{\textbf{x} \in \mathbb{R}^n} \sum_{\gamma \in \Gamma} \chi_{B\left(\gamma; \frac{\kappa}{2}\right)}(\textbf{x}) \geq 1,\\
        B_\Gamma(\kappa) :=& \sup_{\textbf{x} \in \mathbb{R}^n} \sum_{\gamma \in \Gamma} \chi_{B\left(\gamma; \frac{\kappa}{2}\right)}(\textbf{x}) < \infty,
    \end{align*}
    where \( B(\gamma; r) \) represents an open cube in a normed linear space $(\RR^n,\|\cdot\|_{\infty})$ with side-length \( 2r \).
    The notion of a relatively separated set in $\RR$ is referred to as an admissible sequence \cite{delta}.
\end{defin}

\begin{defin}
 The multivariate generalized sampling series associated with a signal $f$ and kernel $\chi$ is defined as
\begin{equation} \label{3.3}
    S_{w}(f)(\textbf{x})=\sum_{\textbf{k}\in \ZZ^n}f\left(\frac{\textbf{k}}{w}\right)\Chi(w\textbf{x}-\textbf{k}), \quad \textbf{x} \in \RR^n, \quad  w>0.
\end{equation}
\end{defin}
The multivariate generalized sampling series \eqref{3.3} does not converge in general.  Therefore, we need to restrict our discussion to an appropriate subspace of $L^{\overrightarrow{P}}(\RR^n)$. Let us define the space $\Delta^{\vc{P}}(\RR^n)$ by
\begin{equation*}
    \Delta^{\vc{P}}(\RR^n):=\left\lbrace  f\in M(\RR^n): \left\lVert(f(x_{\textbf{k}}))\right\rVert_{\ell^{\vc{P}}_w} <\infty \text{ for each relatively separated set } (x_{\textbf{k}})_{\textbf{k} \in \ZZ^n} \right\rbrace.
\end{equation*}
% where
% $$ \|f(x_{\textbf{k}})\|_{\ell^{\overrightarrow{P}}_w} := \bigg(\frac{1}{w}\sum_{k_n\in \ZZ}\cdots\bigg(\frac{1}{w} \sum_{k_2 \in \ZZ}\bigg( \frac{1}{w}\sum_{k_1 \in \ZZ}\left|f\left(x_{k_1},\dots,x_{k_n}\right)\right|^{p_1}\bigg)^{\frac{p_2}{p_1}}\bigg)^{\frac{p_3}{p_2}}\cdots\bigg)^{\frac{1}{p_n}}.$$

In the following proposition concerning the space $\Delta^{\vc{P}}$, we utilize $\|\cdot\|_{\Delta^{\vc{P}}} :\Delta^{\vc{P}} \rightarrow \RR$, given by $\|f\|_{\Delta^{\vc{P}}} = \left\|f \left( \frac{\textbf{k}}{w}\right) \right\|_{\ell^{\overrightarrow P}_w}$. 
    % Moreover, $\|\cdot\|_{X}$ defines a seminorm on the space  $\Delta^{\vc{P}}$.

\begin{prop} \label{p1}
    The space $\Delta^{\vc{P}}(\RR^n)$ is a proper subspace of $L^{\overrightarrow{P}}(\RR^n)$, and $\|\cdot\|_{\Delta^{\vc{P}}}$ is a seminorm on $\Delta^{\vc{P}}(\RR^n)$.
\end{prop}
\begin{proof}
    We begin by demonstrate the result for $n=2$ and then generalize it for any $n \in \NN$.
    We only need to show that $f \in \Delta^{\overrightarrow{P}}$ implies $\displaystyle \left(\int_{\RR}\left(\int_{\RR} |f(x_1,x_2)|^{p_1}dx_1 \right)^{\frac{p_2}{p_1}}dx_2 \right)^{\frac{1}{p_2}}$ is finite. Since $f$ is bounded, it is possible to find $x_{k_1} ^{*}\in [k_1,k_1+1],y_{k_2}^{*}\in [k_2,k_2+1]$, for each $k_1,k_2 \in \ZZ$, such that 
    \[ \esssup_{x_1 \in [k_1,k_1+1], \, x_2 \in [k_2,k_2+1]} |f(x_1, x_2)| < |f(x_{k_1}^{*},y_{k_2}^{*})| + \frac{1}{1+k_1^2+k_2^2},\]
    and it follows that
    \begin{align*}
        \|f\|_{(p_1,p_2)} & =\Bigg(\sum_{k_2 \in\ZZ}\int_{k_2}^{k_2+1}\Bigg(\sum_{k_1 \in \ZZ} \int_{k_1}^{k_1+1}|f(x_1,x_2)|^{p_1}dx_1\Bigg)^{\frac{p_2}{p_1}}dx_2\Bigg)^{\frac{1}{p_2}}\\
		& < \Bigg( \sum_{k_2 \in \ZZ}\Bigg(\sum_{k_1 \in \ZZ}\Bigg(|f(x_{k_1}^{*},y_{k_2}^{*})|+ \frac{1}{1+k_1^2+k_2^2}\Bigg)^{p_1}\Bigg)^{\frac{p_2}{p_1}}\Bigg)^{\frac{1}{p_2}}\\
		& \leq \Bigg(\sum_{k_2 \in \ZZ}\Bigg( \sum_{k_1 \in \ZZ}|f(x_{k_1}^{*},y_{k_2}^{*})|^{p_1}\Bigg)^{\frac{p_2}{p_1}}\Bigg)^{\frac{1}{p_2}}+ \Bigg(\sum_{k_2 \in \ZZ}\Bigg(\sum_{k_1 \in \ZZ} \Bigg( \frac{1}{1+k_1^2+k_2^2}\Bigg)^{p_1}\Bigg)^{\frac{p_2}{p_1}}\Bigg)^{\frac{1}{p_2}}  \\
		&:= S_1+S_2 \hspace{2pt}.
    \end{align*}
    Indeed, $S_2$ is finite for each $p_1,p_2\geq 1$. Concerning $S_1$, the problem is that $(x_{k_1}^{*},y_{k_2}^{*})_ {(k_1,k_2) \in \ZZ^2}$ may not necessarily be a relatively separated in $\RR^2$, but the subsequences $(x_{2k_1}^{*},y_{2k_2}^{*})_ {(k_1,k_2) \in \ZZ^2}$, $(x_{2k_1+1}^{*},y_{2k_2}^{*})_ {(k_1,k_2) \in \ZZ^2}$, $(x_{2k_1}^{*},y_{2k_2+1}^{*})_ {(k_1,k_2) \in \ZZ^2}$, $(x_{2k_1+1}^{*},y_{2k_2+1}^{*})_ {(k_1,k_2) \in \ZZ^2}$ are relatively separated set in $\RR^2$, and hence for each $f \in \Delta^{\overrightarrow{P}}$, we have
    \begin{align*} 
        &\Bigg( \sum_{k_2 \in \mathbb{Z}}\Bigg( \sum_{k_1 \in \mathbb{Z}} |f(x_{k_1}^*, y_{k_2}^*)|^{p_1} \Bigg)^{\frac{p_2}{p_1}}\Bigg)^{\frac{1}{p_2}}\\
        \leq& \Bigg( \sum_{k_2 \in \mathbb{Z}}\Bigg(\sum_{k_1 \in \mathbb{Z}}  |f(x_{2k_1}^*, y_{2k_2}^*)|^{p_1} \Bigg)^{\frac{p_2}{p_1}}\Bigg)^{\frac{1}{p_2}} + \Bigg(\sum_{k_2 \in \mathbb{Z}}\Bigg(\sum_{k_1 \in \mathbb{Z}}   |f(x_{2k_1}^*, y_{2k_2+1}^*)|^{p_1} \Bigg)^{\frac{p_2}{p_1}}\Bigg)^{\frac{1}{p_2}} \\
        &\quad + \Bigg(\sum_{k_2 \in \mathbb{Z}} \Bigg( \sum_{k_1 \in \mathbb{Z}} |f(x_{2k_1+1}^*, y_{2k_2}^*)|^{p_1} \Bigg)^{\frac{p_2}{p_1}}\Bigg)^{\frac{1}{p_2}} + \Bigg( \sum_{k_2 \in \mathbb{Z}}\Bigg(\sum_{k_1 \in \mathbb{Z}}  |f(x_{2k_1+1}^*, y_{2k_2+1}^*)|^{p_1} \Bigg)^{\frac{p_2}{p_1}}\Bigg)^{\frac{1}{p_2}}\\
        <&\, \infty.
    \end{align*}
    In a similar manner, for $n\in \NN$ there exist $2^n$ subsequent relatively separated sets. Using the same reasoning, since  $f\in \Delta^{\vc{P}}(\RR^n)$, it follows that the all the mixed norm Lebesgue space are finite. Hence, $f\in L^{\vc{P}}(\RR^n)$. 
    For the second part, we provide the proof in the following steps:
    \begin{enumerate}[label=(\roman*), left=0pt]
        % \item [(i)] According to the definition of $\|f(x_{\textbf{k}})\|_{\ell^{\overrightarrow{P}}_w}$, it is clear that it is non-negative. Consider \[ f(\textbf{u}) = 
        % \begin{cases}
        %     0, & \text{if } \textbf{u} = \textbf{k} \text{ for } \textbf{k} \in \mathbb{Z}^n,\\
        %     1, &  otherwise,
        % \end{cases}\]
        % we have  $\|f(x_{\textbf{k}})\|_{\ell^{\overrightarrow{P}}_w}=0,$ but $f$ is non-zero.

        \item For any $f\in \Delta^{\vc{P}}(\RR^n)$, we have 
       % $ p(\alpha f) = |\alpha| \hspace{2pt} p(f)$
        $\|\alpha f\|_{\Delta^{\vc{P}}}= |\alpha|  \hspace{2pt} \|f\|_{\Delta^{\vc{P}}}$,  
        $ \forall \alpha \in \RR$. 
        
        \item For any $f, g\in \Delta^{\vc{P}}(\RR^n)$ and utilizing triangle inequality, we have 
        % $$p(f+g) \leq p(f) + p(g).$$
        $$\|f+g\|_{\Delta^{\vc{P}}} \leq \|f\|_{\Delta^{\vc{P}}}+\|g\|_{\Delta^{\vc{P}}}.$$
    \end{enumerate}
    Therefore, $\|\cdot\|_{\Delta^{\vc{P}}} $ is a seminorm on $\Delta^{\vc{P}}(\RR^n)$. This completes the proof.
\end{proof}

By restricting the allowable functions \( f \) in (\ref{3.3}) to the space \( \Delta^{\overrightarrow{P}} \), the operators \( (S_{w})_{w>0} \) becomes a bounded operator from \( \Delta^{\overrightarrow{P}}  \) into \( L^{\overrightarrow{P}} \).

\begin{thm} \label{interpol}
    The generalized sampling operator $S_{w}:\Delta^{\overrightarrow{P}} \left(\subset  L^{\overrightarrow{P}}(\RR^{n})\right)\rightarrow L^{\overrightarrow{P}}(\RR^{n}) $ satisfies
    \begin{equation*}
		\|S_{w}f\|_{\overrightarrow P}\leq (m_0(\Chi))^{1-\frac{1}{p_n}}\hspace{2pt} \|\Chi\|_1^{\frac{1}{p_{n}}} \left\|f \left( \frac{\textbf{k}}{w}\right) \right\|_{\ell^{\overrightarrow P}_w}.
    \end{equation*}
\end{thm}

\begin{proof}
    To begin, we verifying the result for $n=2.$ Consider $ \overrightarrow{P}=(p_1,p_2) \in [1,\infty)^2$ and $ f \in L^{\overrightarrow{P}}(\RR^{2}).$
    From convexity property on the summation and Jensen's inequality on integration, we get 
    \begin{align*}
		\int_{\RR} \bigg|\sum_{k_2 \in \ZZ}\sum_{k_1 \in \ZZ}  &\Chi(wx_1-k_1,wx_2-k_2) f\left(\frac{k_1}{w},\frac{k_2}{w}\right)\bigg|^{p_1}dx_1\\
		\leq &\,(m_{0}(\Chi))^{p_1}\int_{\RR} \sum_{k_2 \in \ZZ}\sum_{k_1 \in \ZZ} \frac{\Chi(wx_1-k_1,wx_2-k_2)}{m_{0}(\Chi)} \left| f\left(\frac{k_1}{w},\frac{k_2}{w}\right)\right|^{p_1} dx_1\\
		\leq &\, \frac{(m_{0}(\Chi))^{p_1-1}}{w}\sum_{k_2 \in \ZZ}\sum_{k_1 \in \ZZ} \|\Chi(\cdot,wx_2-k_2)\|_1 \left| f\left(\frac{k_1}{w},\frac{k_2}{w}\right)\right|^{p_1}.
    \end{align*}
    Since $p_1\leq p_2$, therefore using the convexity and Jensen's inequality on integration, we get
    \begin{align*}
		&\|S_wf\|_{(p_1,p_2)}^{p_2}\\
		\leq&\, m_{0}(\Chi)^{p_2-\frac{p_2}{p_1}} \int_{\RR} \bigg(\sum_{k_2 \in \ZZ} \|\Chi(\cdot,wx_{2}-k_2)\|_1\bigg)^{\frac{p_2}{p_1}} \\
		&\hspace{1.5in}\times  \bigg(\frac{1}{w} \sum_{k_2 \in \ZZ}\frac{\|\Chi(\cdot,wx_{2}-k_2)\|_1}{\sum_{k_2 \in \ZZ} \|\Chi(\cdot,wx_{2}-k_2)\|_1} \sum_{k_1 \in \ZZ}\left| f\left(\frac{k_1}{w},\frac{k_2}{w}\right)\right|^{p_1}\bigg)^{\frac{p_2}{p_1}}\, dx_2\\
		\leq &\, m_{0}(\Chi)^{p_2-\frac{p_2}{p_1}}\int_{\RR} \bigg(\sum_{k_2 \in \ZZ} \|\Chi(\cdot,wx_{2}-k_2)\|_1\bigg)^{\frac{p_2}{p_1}-1}\\
		&\hspace{1.5in}\times \sum_{k_2 \in \ZZ}\|\Chi(\cdot,wx_{2}-k_2)\|_1 \bigg( \frac{1}{w}\sum_{k_1 \in \ZZ}\left| f\left(\frac{k_1}{w},\frac{k_2}{w}\right)\right|^{p_1}\bigg)^{\frac{p_2}{p_1}}\, dx_2.
    \end{align*}
    Since the absolute moment of the kernel $\Chi$ of order $0$ is finite. Therefore, for every $1\leq j\leq n$, we have
    \begin{align*}
    	&\sum_{k_{n} \in \ZZ} \cdots \sum_{k_j \in \ZZ} \|\Chi(\cdot,\cdots,\cdot,wx_{j}-k_j,\dots, wx_n-k_n)\|_{1}\\
    	\leq& \int_{[0,1]^{j-1}} \sum_{\textbf{k}\in \ZZ^n} |\Chi(x_1-k_1,x_2-k_2,\dots,wx_n-k_n)|dx_1dx_2\cdots dx_{j-1}\\
    	\leq& \sup_{\textbf{x}\in \RR^n} \sum_{\textbf{k}\in \ZZ^n} |\Chi(\textbf{x}-\textbf{k})|=m_0(\Chi).
    \end{align*}
    
    Hence,
    \begin{align*}
		\|S_wf\|_{(p_1,p_2)}^{p_2}\leq &\, \frac{m_0(\Chi)^{p_2-1}}{w}\|\Chi\|_{1}\sum_{k_2\in \ZZ}  \bigg( \frac{1}{w}\sum_{k_1 \in \ZZ}\left| f\left(\frac{k_1}{w},\frac{k_2}{w}\right)\right|^{p_1}\bigg)^{\frac{p_2}{p_1}}\, \\
		\leq & \,m_0(\Chi)^{p_2-1}\|\Chi\|_{1} \left\|f \left(\frac{k_1}{w},\frac{k_2}{w}\right)\right\|_{\ell^{\overrightarrow P}_w}^{p_2}\\
		\|S_wf\|_{(p_1,p_2)}\leq & (m_{0}(\Chi))^{1-\frac{1}{p_2}} \|\Chi\|_{1}^{\frac{1}{p_2}} \left\|f \left(\frac{k_1}{w},\frac{k_2}{w}\right)\right\|_{\ell^{\overrightarrow P}_w}.
    \end{align*}
    %%%%%%%%%%%%%%%%%%%%%%%%%%%%%%%%%%%%%%%%%%%%%
    Now following the similar line of proof for $\vc{P}=(p_1,p_2,\dots,p_n)$, we get
    \begin{flalign*}
		\|S_{w}f\|_{\overrightarrow P}\leq (m_0(\Chi))^{1-\frac{1}{p_n}}\hspace{2pt} \|\Chi\|_1^{\frac{1}{p_{n}}} \left\|f\left(\frac{\textbf{k}}{w}\right) \right\|_{\ell^{\overrightarrow P}_w}.
    \end{flalign*}
    This completes the proof.
\end{proof}

% We have established that the operators $(S_w)_{w>0}$ are a bounded linear operator from $\Delta^{\overrightarrow{P}}$ to $L^{\overrightarrow{P}}$.

In the next section, we will discuss the well-known family $(K_w)_{w>0}$ known as Kantorovich-type sampling operators. These operators plays a significant role in the approximation of not necessarily continuous signals, with particular importance in multivariate contexts. Moreover, this family of operators has been proven to be especially effective in the field of image processing.

\section{Kantorovich-type sampling operators on mixed norm Lebesgue space} \label{4}
In this section, we investigate the boundedness and convergence of Kantorovich-type sampling operators within mixed norm Lebesgue space. It is noteworthy that the convergence of $(K_w)_{w>0}$ follows from the fact that $C_{c}(\RR^n)$ is dense in $L^{\overrightarrow{P}}(\RR^{n})$. First we proceed to demonstrate that for a fixed $w>0,$ the operator $K_{w}: L^{\overrightarrow{P}}(\RR^{n})\rightarrow L^{\overrightarrow{P}}(\RR^{n}) $ is a bounded linear operator.

\begin{defin}  
    Let $f:\RR^{n}\rightarrow\RR$ be a locally integrable function. For $\textbf{x} \in \RR^n$, the Kantorovich-type sampling series associated with $f$ is given by \cite{multi}  
    \begin{equation}\label{kan}
        K_{w}(f)(\textbf{x})=\sum_{\textbf{k}\in \ZZ^n}\Chi(w\textbf{x}-\textbf{k}) \hspace{3pt} w^{n} \int_{I_{\textbf{k},w}} f(\textbf{t})d\textbf{t} 
    \end{equation}
    where $I_{\textbf{k},w}=\prod\limits_{j=1}^{n}\left[\frac{k_j}{w},\frac{k_j+1}{w}\right].$
\end{defin} 

% As \( w \to \infty \), the operators \( (K_w)_{w>0} \) are proficient in pointwise reconstruction of continuous and bounded signals. 

%We aim to establish the convergence of Kantorovich-type sampling operators in mixed norm Lebesgue spaces $L^{\overrightarrow{P}}(\RR^{n})$.  To this end, we utilize both  the boundedness and the convergence in $C_c(\RR^n)$ of these operators in the setting of mixed norm Lebesgue spaces.
%In this direction, we first prove their boundedness within $L^{\overrightarrow{P}}(\RR^{n})$.
\begin{thm} \label{thm 4.2}
    The Kantorovich-type sampling operator $K_{w}: L^{\overrightarrow{P}}(\RR^{n})\rightarrow L^{\overrightarrow{P}}(\RR^{n}) $ satisfies 
    \begin{equation*}
		\|K_{w}f\|_{\overrightarrow{P}}\leq (m_0(\Chi))^{1-\frac{1}{p_n}}\hspace{2pt} \|\Chi\|_1^{\frac{1}{p_{n}}}\hspace{2pt} \left\|f \right\|_{\overrightarrow{P}}.
    \end{equation*}
\end{thm}

\begin{proof}
    We first examine this result for $n=2,$ which will further be extend to any $n \in \NN.$ Now using (\ref{kan}), we obtain
    \begin{align*} 
        &\|K_wf\|_{(p_1,p_2)}^{p_2}\\	
		=&\int_{\RR}\bigg(\int_{\RR} \bigg| \sum_{k_2 \in \ZZ}\sum_{k_1 \in \ZZ} \Chi(wx_1-k_1,wx_2-k_2) w^{2} \int_{\frac{k_2}{w}}^{\frac{k_2+1}{w}}\int_{\frac{k_1}{w}}^{\frac{k_1+1}{w}}f(t_1,t_2)dt_1dt_2\bigg|^{p_1}dx_1\bigg)^{\frac{p_2}{p_1}}dx_2.
    \end{align*}
    From convexity property on the summation and Jensen's inequality on integration, we get 
    \begin{align*}
        &\int_{\RR} \bigg|\sum_{k_2 \in \ZZ}\sum_{k_1 \in \ZZ}  \Chi(wx_1-k_1,wx_2-k_2) w^{2} \int_{\frac{k_2}{w}}^{\frac{k_2+1}{w}} \int_{\frac{k_1}{w}}^{\frac{k_1+1}{w}}|f(t_1,t_2)|dt_1dt_2\bigg|^{p_1}dx_1\\
        \leq &(m_{0}(\Chi))^{p_1}\int_{\RR} \sum_{k_2 \in \ZZ}\sum_{k_1 \in \ZZ} \frac{\Chi(wx_1-k_1,wx_2-k_2)}{m_{0}(\Chi)} \Big( w^{2} \int_{\frac{k_2}{w}}^{\frac{k_2+1}{w}} \int_{\frac{k_1}{w}}^{\frac{k_1+1}{w}}|f(t_1,t_2)|dt_1dt_2 \Big)^{p_1}\\
        \leq &(m_{0}(\Chi))^{p_1-1}\sum_{k_2 \in \ZZ}\sum_{k_1 \in \ZZ} \|\Chi(\cdot,wx_2-k_2)\|_1 \int_{\frac{k_2}{w}}^{\frac{k_2+1}{w}}\int_{\frac{k_1}{w}}^{\frac{k_1+1}{w}} w|f(t_1,t_2)|^{p_1}dt_1dt_2\\
        \leq &(m_{0}(\Chi))^{p_1-1}\sum_{k_2 \in \ZZ} \|\Chi(\cdot,wx_2-k_2)\|_1 \int_{\frac{k_2}{w}}^{\frac{k_2+1}{w}}\int_{\RR} w|f(t_1,t_2)|^{p_1}dt_1dt_2.
    \end{align*}
    Since $p_1\leq p_2$, therefore using the convexity and Jensen's inequality on integration, we get
    \begin{align*}
		&\|K_wf\|_{(p_1,p_2)}^{p_2}\\
		\leq&\, m_{0}(\Chi)^{p_2-\frac{p_2}{p_1}}\int_{\RR} \bigg(\sum_{k_2 \in \ZZ} \|\Chi(\cdot,wx_{2}-k_2)\|_1\bigg)^{\frac{p_2}{p_1}} \\
		&\hspace{1.5in}\times \bigg( \sum_{k_2 \in \ZZ}\frac{\|\Chi(\cdot,wx_{2}-k_2)\|_1}{\sum_{k_2 \in \ZZ} \|\Chi(\cdot,wx_{2}-k_2)\|_1} w \int_{\frac{k_2}{w}}^{\frac{k_2+1}{w}} \int_{\RR} |f(t_1,t_2)|^{p_1}dt_1 dt_2\bigg)^{\frac{p_2}{p_1}}\, dx_2\\
		\leq &\, m_{0}(\Chi)^{p_2-\frac{p_2}{p_1}}\int_{\RR} \bigg(\sum_{k_2 \in \ZZ} \|\Chi(\cdot,wx_{2}-k_2)\|_1\bigg)^{\frac{p_2}{p_1}-1} \\
		&\hspace{1.5in} \times\sum_{k_2 \in \ZZ}\|\Chi(\cdot,wx_{2}-k_2)\|_1 \bigg( w \int_{\frac{k_2}{w}}^{\frac{k_2+1}{w}} \int_{\RR} |f(t_1,t_2)|^{p_1}dt_1 dt_2 \bigg)^{\frac{p_2}{p_1}}\, dx_2\\
		\leq &\, m_0(\Chi)^{p_2-1}\|\Chi\|_{1}\sum_{k_2\in \ZZ} \int_{\frac{k_2}{w}}^{\frac{k_2+1}{w}} \bigg( \int_{\RR} |f(t_1,t_2)|^{p_1}dt_1 \bigg)^{\frac{p_2}{p_1}}\, dt_2\\
		\leq & \,m_0(\Chi)^{p_2-1}\|\Chi\|_{1}\|f\|_{(p_1,p_2)}^{p_2}.
    \end{align*}
    This gives $$\|K_wf\|_{(p_1,p_2)}\leq (m_{0}(\Chi))^{1-\frac{1}{p_2}} \|\Chi\|_{1}^{\frac{1}{p_2}}\|f\|_{(p_1,p_2)}.$$ Proceeding in a similar manner, we obtain
    \begin{flalign*}
        \|K_{w}f\|_{\overrightarrow{P}}& \leq (m_0(\Chi))^{1-\frac{1}{p_n}}\|\Chi\|_1^{\frac{1}{p_{n}}}\|f \|_{\overrightarrow{P}}.
    \end{flalign*}
\end{proof}

For a given $f\in L^{\vc{P}}(\RR^n)$, the Kantorovich-type sampling operator $K_w$ acting on $f$ approximates the function $f$ in $L^{\vc{P}}(\RR^n)$. In order to prove this result, we first observe that $C_c(\RR^n)$ is dense in $L^{\vc{P}}(\RR^n)$ and $K_wf\to f$ as $w\to \infty$ for every $f\in C_c(\RR^n)\subset L^{\vc{P}}(\RR^n)$.
% The following lemma will be useful for the subsequent analysis.
In order to prove the convergence result in $L^{\overrightarrow{P}}(\RR^n)$, the following lemma on the absolute continuity of $\Chi$ plays a key role.

\begin{lemma} \label{cc}
Let $\Chi$ be the kernel. For every ${c}>0$ and $\epsilon >0$, there exists a constant $L>0$ such that
	\begin{equation*}
	\displaystyle \int_{\|\textbf{x}\|_{\infty}>{L}}w^n|\Chi(w\textbf{x}-\textbf{k})|d\textbf{x} < \epsilon  
	\end{equation*}
	for sufficiently large $w>0$ and  $\textbf{k} \in \ZZ^n$ such that $\displaystyle {\|\textbf{k}\|_{\infty}} \leq  wc$.
\end{lemma}
	\begin{proof}
		Since $\Chi \in L^1(\RR^n)$, there exists a constant $M>0$ such that  
		\begin{equation*}
\int_{{\|\textbf{u}\|_{\infty}}>M }|\Chi(\textbf{u})|d\textbf{u}  <\epsilon.
\end{equation*}
For $c>0, w \geq 1$, consider $\textbf{k}$ such that $\displaystyle {\|\textbf{k}\|_{\infty}} \leq wc$. Further we can  also find $L>0$ such that $w(L-c) >M$. The detailed proof can be found in \cite{multi}.
\end{proof}

In the following theorem, we demonstrate the convergence of the Kantorovich-type sampling operator in $C_{c}(\RR^{n})$ within the framework of mixed norm Lebesgue space.
\begin{thm} \label{thm 4.3}
	For $f\in C_{c}(\RR^{n})$, it holds that
	\begin{align*}
		\lim_{w\to\infty}\|K_{w}f-f\|_{\overrightarrow P}=0.
		\end{align*}
	\end{thm}
\begin{proof}
We proceed to analyze the result for $n=2$, i.e., for $f \in  C_{c}(\RR^{2})$, and afterwards we will generalize it for any $n \in \NN.$  We will prove that	$$\lim_{w\to\infty}\|K_{w}f-f\|_{(p_1,p_2)}=0.$$ We will be using the \emph{Vitali Convergence theorem} in mixed norm Lebesgue space \cite{vitali}. We outline the proof in the following steps:\par

\textbf{Step 1.} In view of  \cite[Theorem 4.1]{multi}, we have	$$\lim_{w\to\infty}\|K_{w}f-f\|_{\infty}=0.$$\par 

\textbf{Step 2.}  As $f\in C_c(\RR^n)$, there exists $a>0$ such that $\text{supp}\, f \subseteq [- a,  a]^2$. 
% Taking $a> \bar a+\delta$ for some positive $\delta$, we consider $ [-wa,wa]^2 $. 
For every sufficiently large $w>0$,   whenever $(k_1,k_2) \notin [-wa-1,wa]^2$, we have $\left\lbrace [\frac{k_1}{w},\frac{k_1+1}{w}] \times [\frac{k_2}{w},\frac{k_2+1}{w}]\right\rbrace \cap [- a, a]^2=\emptyset$.
% $(k_1,k_2) \notin [-wa,wa]^2$, we have $\left\lbrace [\frac{k_1}{w},\frac{k_1+1}{w}] \times [\frac{k_2}{w},\frac{k_2+1}{w}]\right\rbrace \cap [- \bar a, \bar a]^2=\emptyset$.
This gives
\begin{equation*} \label{4.4}
    \int_{\frac{k_2}{w}}^{\frac{k_2+1}{w}}	\int_{\frac{k_1}{w}}^{\frac{k_1+1}{w}}|f(t_1,t_2)|dt_1dt_2=0.
\end{equation*} 
By Lemma \ref{cc}, for every fixed $\epsilon >0$, there exists a constant $M>0$ (we can assume $M >  a$) such that 
\begin{equation*}
    \int_{\|\textbf{x}\|_{\infty}>{M}}w^n|\Chi(w\textbf{x}-\textbf{k})|d\textbf{x} < \epsilon
\end{equation*}
for every $(k_1,k_2) \in [-wa-1,wa]^2$. For a set $\textbf{G}=[-M,M]^2$, we compute 
\begin{align}\label{j}
    I &= \|K_wf\|^{p_2}_{L^{(p_1,p_2)}(\textbf{G}^c)} \nonumber\\
    &\leq \bigg( \int_{|x_2|>M}\bigg(\int_{|x_1|\leq M}|(K_{w}f)(x_1,x_2)|^{p_1}dx_1\bigg)^{\frac{p_2}{p_1}}dx_2\bigg)\nonumber\\
    & \hspace{1in}+\bigg(\int_{|x_2|\leq M}\bigg(\int_{|x_1|>M}|(K_{w}f)(x_1,x_2)|^{p_1}dx_1\bigg)^{\frac{p_2}{p_1}}dx_2\bigg)\nonumber\\
    & \hspace{2in}+\bigg(\int_{|x_2|>M}\bigg(\int_{|x_1|>M}|(K_{w}f)(x_1,x_2)|^{p_1}dx_1\bigg)^{\frac{p_2}{p_1}}dx_2\bigg)\nonumber\\
    &:= I_1 +I_2 +I_3.
\end{align}
First, we compute $I_1$. Utilizing Jensen's inequality twice and Fubini-Tonelli theorem, we have 
\begin{flalign*}
    &\int_{|x_1|\leq M}\Bigg(  \sum_{\frac{k_2}{w} \in [-a-\frac{1}{w}, a]}\sum_{\frac{k_1}{w}\in  [-a-\frac{1}{w}, a]}|\Chi(wx_1-k_1,wx_2-k_2)|w^2\int_{\frac{k_2}{w}}^{\frac{k_2+1}{w}}\int_{\frac{k_1}{w}}^{\frac{k_1+1}{w}}|f(t_1,t_2)|dt_1dt_2\Bigg)^{p_1}dx_1\\
    \leq& \left(m_0(\Chi)\right)^{p_1-1} \sum_{\frac{k_2}{w}\in[-a-\frac{1}{w}, a]} \sum_{\frac{k_1}{w} \in [-a-\frac{1}{w}, a]}  w^2 \int_{\frac{k_2}{w}}^{\frac{k_2+1}{w}}\int_{\frac{k_1}{w}}^{\frac{k_1+1}{w}}|f(t_1,t_2)|^{p_1}dt_1dt_2 \\
    &\hspace{3 in}\times\int_{|x_1|\leq M} |\Chi(wx_1-k_1,wx_2-k_2)|dx_1\\
    \leq&  \left(m_0(\Chi)\right)^{p_1-1} \sum_{\frac{k_2}{w}\in  [-a-\frac{1}{w}, a]} \|\Chi(\cdot,wx_2-k_2)\|_{L^{1}([-M,M])}\ w \int_{\frac{k_2}{w}}^{\frac{k_2+1}{w}}\int_{-a-\frac{1}{w}}^{a+\frac{1}{w}}|f(t_1,t_2)|^{p_1}dt_1dt_2.
    % <&  \left(m_0(\Chi)\right)^{p_1-1} \sum_{\frac{k_2}{w}\in[-a-\frac{1}{w}, a]} \|\Chi(\cdot,wx_2-k_2)\|_{L^{1}([-M,M])}\ w \int_{\frac{k_2}{w}}^{\frac{k_2+1}{w}}\int_{-a}^{a+\delta'}|f(t_1,t_2)|^{p_1}dt_1dt_2.
\end{flalign*}
 % Last step follows form the Archimedean property, i.e., for any $\delta', w > 0$, it holds that $\frac{1}{w} < \delta'$.
Again using Jensen's inequality and Fubini-Tonelli theorem, we get 
\begin{flalign*}
    I_1	\leq& \, \left(m_0(\Chi)\right)^{p_2-\frac{p_2}{p_1}} \int_{|x_2|>M}\Bigg( \displaystyle\sum_{\frac{k_2}{w}\in[-a-\frac{1}{w}, a]} \|\Chi(\cdot,wx_2-k_2)\|_{L^{1}([-M,M])} \Bigg)^{\frac{p_2}{p_1}-1}\\
    &\hspace{0.1in} \times  \sum_{\frac{k_2}{w}\in[-a-\frac{1}{w}, a]} \|\Chi(\cdot,wx_2-k_2)\|_{L^1([-M,M])}\ \left(w \int_{\frac{k_2}{w}}^{\frac{k_2+1}{w}}\int_{-a-\frac{1}{w}}^{a+\frac{1}{w}}|f(t_1,t_2)|^{p_1}dt_1dt_2\right)^{\frac{p_2}{p_1}}dx_2\\
    \leq & \left(m_0(\Chi)\right)^{p_2-1}\sum_{\frac{k_2}{w}\in[-a-\frac{1}{w}, a]}\ \int_{\frac{k_2}{w}}^{\frac{k_2+1}{w}} \left(\int_{-a-\frac{1}{w}}^{a+\frac{1}{w}}|f(t_1,t_2)|^{p_1}dt_1\right)^{\frac{p_2}{p_1}}dt_2\\
    &\hspace{2 in} \times \int_{|x_2|>M} w\|\Chi(\cdot,wx_2-k_2)\|_{L^1([-M,M])}dx_2.
% \leq & \,  \left(m_0(\Chi)\right)^{p_2-1}\|\Chi\|_{L^{1,1}([-M,M]^c \times [-M,M]^c)}\int_{-c}^{c} \left( \int_{-c}^{c}|f(t_1,t_2)|^{p_2}dt_1\right)^{\frac{p_2}{p_1}} dt_2\\
\end{flalign*}
Thus
        \begin{flalign*}
		I_1 <& \, \frac{\epsilon}{3}\hspace{2pt}  \left(m_0(\Chi)\right)^{p_2-1}\|f\|^{p_2}_{(p_1,p_2)}.
        \end{flalign*}
        Note that the last inequality follows from the following observation:
        \begin{equation*}
         \int_{|x_2|>M}  w \|\Chi(\cdot,wx_2-k_2)\|_{L^1([-M,M])} \hspace{2pt} dx_2< \frac{\epsilon}{3}.
        \end{equation*}
Similarly, we can compute $I_2$. Thus we get
    $$I_2 < \frac{\epsilon}{3}\hspace{2pt}  \left(m_0(\Chi)\right)^{p_2-1}\|f\|^{p_2}_{(p_1,p_2)}. $$
   One can wite $I_3$ as follows:
 	\begin{multline*}
	I_3:=\Bigg(\int_{|x_2|>M}\Bigg(\int_{|x_1|>M}\Bigg(\sum_{\frac{k_2}{w} \in [-a-\frac{1}{w}, a]}\sum_{\frac{k_1}{w}\in[-a-\frac{1}{w}, a]}|\Chi(wx_1-k_1,wx_2-k_2)|\\
		w^{2}\int_{\frac{k_2}{w}}^{\frac{k_2+1}{w}}\int_{\frac{k_1}{w}}^{\frac{k_1+1}{w}}|f(t_1,t_2)|dt_1dt_2\Bigg)^{p_1}dx_1\Bigg)^{\frac{p_2}{p_1}}dx_2\Bigg).
	\end{multline*}
    In view of Jensen's inequality and Fubini-Tonelli theorem, we have 	\begin{flalign*}
		&\int_{|x_1|>M}\Bigg(  \sum_{\frac{k_2}{w} \in [-a-\frac{1}{w}, a]}\sum_{\frac{k_1}{w}\in[-a-\frac{1}{w}, a]}|\Chi(wx_1-k_1,wx_2-k_2)|w^2\int_{\frac{k_2}{w}}^{\frac{k_2+1}{w}}\int_{\frac{k_1}{w}}^{\frac{k_1+1}{w}}|f(t_1,t_2)|dt_1dt_2\Bigg)^{p_1}dx_1\\
		\leq& \left(m_0(\Chi)\right)^{p_1-1} \sum_{\frac{k_2}{w}\in[-a-\frac{1}{w}, a]} \sum_{\frac{k_1}{w} \in [-a-\frac{1}{w}, a]} \Bigg(w^2 \int_{\frac{k_2}{w}}^{\frac{k_2+1}{w}}\int_{\frac{k_1}{w}}^{\frac{k_1+1}{w}}|f(t_1,t_2)|dt_1dt_2\Bigg)^{p_1}\\
		&\hspace{3 in}\times\int_{|x_1|>M}|\Chi(wx_1-k_1,wx_2-k_2)| dx_1\\
		\leq& \left(m_0(\Chi)\right)^{p_1-1} \sum_{\frac{k_2}{w}\in[-a-\frac{1}{w}, a]} \sum_{\frac{k_1}{w} \in [-a-\frac{1}{w}, a]}  w^2 \int_{\frac{k_2}{w}}^{\frac{k_2+1}{w}}\int_{\frac{k_1}{w}}^{\frac{k_1+1}{w}}|f(t_1,t_2)|^{p_1}dt_1dt_2 \\
		&\hspace{3 in}\times\int_{|x_1|>M} |\Chi(wx_1-k_1,wx_2-k_2)|dx_1\\
		\leq&  \left(m_0(\Chi)\right)^{p_1-1} \sum_{\frac{k_2}{w}\in[-a-\frac{1}{w}, a]} \|\Chi(\cdot,wx_2-k_2)\|_{L^1([-M,M]^c)}\ w \int_{\frac{k_2}{w}}^{\frac{k_2+1}{w}}\int_{-a-\frac{1}{w}}^{a+\frac{1}{w}}|f(t_1,t_2)|^{p_1}dt_1dt_2.
	\end{flalign*}
	Again using Jensen's inequality and Fubini-Tonelli theorem, we obtain
	\begin{flalign*}
		I_3
		\leq& \, \left(m_0(\Chi)\right)^{p_2-\frac{p_2}{p_1}} \int_{|x_2|>M} \Bigg( \displaystyle\sum_{\frac{k_2}{w}\in[-a-\frac{1}{w}, a]} \|\Chi(\cdot,wx_2-k_2)\|_{L^1([-M,M]^c)} \Bigg)^{\frac{p_2}{p_1}-1}\\
		&\hspace{0.1 in} \times \sum_{\frac{k_2}{w}\in[-a-\frac{1}{w}, a]} \|\Chi(\cdot,wx_2-k_2)\|_{L^1([-M,M]^c)}\ \left(w \int_{\frac{k_2}{w}}^{\frac{k_2+1}{w}}\int_{-a-\frac{1}{w}}^{a+\frac{1}{w}}|f(t_1,t_2)|^{p_1}dt_1dt_2\right)^{\frac{p_2}{p_1}} dx_2\\
        		\leq & \left(m_0(\Chi)\right)^{p_2-1}\sum_{\frac{k_2}{w}\in[-a-\frac{1}{w}, a]}\ \int_{\frac{k_2}{w}}^{\frac{k_2+1}{w}} \left(\int_{-a-\frac{1}{w}}^{a+\frac{1}{w}}|f(t_1,t_2)|^{p_1}dt_1\right)^{\frac{p_2}{p_1}}dt_2\\
                &\hspace{2 in} \times \int_{|x_2|>M} w\|\Chi(\cdot,wx_2-k_2)\|_{L^1([-M,M]^c)}dx_2.
        \end{flalign*}
      Using  Lemma \ref{cc}, we get
        \begin{flalign*}
		I_3< & \, \frac{\epsilon}{3}\hspace{2pt}  \left(m_0(\Chi)\right)^{p_2-1}\|f\|^{p_2}_{(p_1,p_2)}.
\end{flalign*}
        From (\ref{j}), we get
     \begin{flalign*}
         I &< \frac{ \epsilon}{3}\hspace{2pt}  \left(m_0(\Chi)\right)^{p_2-1}\|f\|^{p_2}_{(p_1,p_2)} + \frac{ \epsilon}{3}\hspace{2pt}  \left(m_0(\Chi)\right)^{p_2-1}\|f\|^{p_2}_{(p_1,p_2)} +\frac{ \epsilon}{3}\hspace{2pt}  \left(m_0(\Chi)\right)^{p_2-1}\|f\|^{p_2}_{(p_1,p_2)} \\
         % &\leq \frac{ \epsilon}{3}\hspace{2pt}  \left(m_0(\Chi)\right)^{p_2-1}\|f\|^{p_2}_{(p_1,p_2)}\\
         &< \epsilon\hspace{2pt}  \left(m_0(\Chi)\right)^{p_2-1}\|f\|^{p_2}_{(p_1,p_2)} .
     \end{flalign*}\par

    \textbf{Step 3.} For every $\epsilon >0$, $\exists \hspace{3pt} \delta >0$ such that  
    \begin{equation}\label{leb1}
    \displaystyle\left(\int_{B_2} \int_{B_1}|\Chi(u_1,u_2)|du_1du_2 \right) <  \displaystyle\frac{\epsilon ^{p_2}}{(m_{0}(\Chi))^{p_2-1} \|f\|^{p_2}_{(p_1,p_2)}}\end{equation}
    for every pair of measurable sets $B_i \subset \RR$ with $\displaystyle \mu(B_i)<\delta $, for $i=1,2.$
    By following the approach used in the proof of Theorem \ref{thm 4.2} and absolute continuity property of the Lebesgue integral (\ref{leb1}), we get
    \begin{flalign*}
        \left( \int_{B_2}\left(\int_{B_1}|(K_{w}f)(x_1,x_2)|^{p_1}dx_1\right)^{\frac{p_2}{p_1}}dx_2 \right)^{\frac{1}{p_2}} & \leq (m_{0}(\Chi))^{1-\frac{1}{p_2}} \|f\|_{(p_1,p_2)}  \|\Chi\|^{\frac{1}{p_2}}_{L^{(1,1)}(B_1 \times B_2)}\\
        & < \epsilon.
		% &< \epsilon^{\frac{1}{p_2}} (m_{0}(\Chi))^{1-\frac{1}{p_2}} \|f\|_{(p_1,p_2)}.
    \end{flalign*}
    Thus, the integrals $\Bigg(\displaystyle \int_{B_2}\Bigg(\int_{B_1}|K_wf|^{p_1}dx_1\Bigg)^{\frac{p_2}{p_1}}dx_2\Bigg)^{\frac{1}{p_2}}$	are equi-absolutely continuous for arbitrary choices of $B_1,B_2$.

    Hence, by using the Vitali convergence theorem, we get $\|K_{w}f-f\|_{(p_1,p_2)}\to 0$ as $w\to \infty$. Following along the similar lines, one can generalize the result for $n \geq 3.$
\end{proof}

The space of compactly supported continuous functions $C_c(\mathbb{R}^n)$ is dense in the mixed norm Lebesgue spaces $L^{\overrightarrow{P}}(\mathbb{R}^n)$ (see \cite{ivec}). Therefore, Theorem \ref{thm 4.3} can be generalized to a mixed norm Lebesgue space $L^{\overrightarrow P}(\RR^{n})$. In particular, we have the following result.

% \begin{lemma} \label{cd}
% 	\cite{ivec} The space $C_{c}(\mathbb{R}^n)$ is densely contained within $L^{\overrightarrow P}(\RR^n)$, for $\overrightarrow P\in [1,\infty)^{n}.$
% \end{lemma}

\begin{thm}
    For $f\in L^{\overrightarrow P}(\RR^{n})$, we have $$ \lim_{w\to\infty}\|K_{w}f-f\|_{\overrightarrow{P}}=0. $$
\end{thm}
\begin{proof}
 Let $f\in L^{\overrightarrow{P}}(\RR^{n})$ and $\epsilon >0$ be fixed. Using denisty of $C_{c}(\mathbb{R}^n)$ in $L^{\overrightarrow P}(\RR^n)$, 
 % $\epsilon$ is chosen according to Lemma \ref{cd}, which states that 
 there exists a function $g \in C_{c}(\RR^{n})$ such that 
	$$\|g-f\|_{\overrightarrow{P}}< \frac{\epsilon}{2\left(1+(m_{0}(\Chi))^{1-\frac{1}{p_n}}\hspace{3pt}\|\Chi\|_{1}^{\frac{1}{p_n}}\right)}.$$
	Utilizing the triangle inequality for $\|\cdot\|_{\overrightarrow{P}}$ and Theorem \ref{thm 4.2} and Theorem \ref{thm 4.3}, it follows that
	\begin{equation*}
		\begin{split}
			\|K_{w}f-f\|_{\overrightarrow{P}} &\leq \|g-f\|_{\overrightarrow{P}}+\|K_{w}g-g\|_{\overrightarrow{P}}+\|K_{w}f-K_{w}g\|_{\overrightarrow{P}}\\
			&\leq \|g-f\|_{\overrightarrow{P}}+\|K_{w}g-g\|_{\overrightarrow{P}}+(m_{0}(\Chi))^{1-\frac{1}{p_n}}\hspace{3pt}\|\Chi\|_{1}^{\frac{1}{p_n}}\hspace{3pt}\|g-f \|_{\overrightarrow{P}} \\
			& \leq\left(1+(m_{0}(\Chi))^{1-\frac{1}{p_n}}\hspace{3pt}\|\Chi\|_{1}^{\frac{1}{p_n}}\right)\|g-f \|_{\overrightarrow{P}} + \|K_{w}g-g\|_{\overrightarrow{P}}\\
			& <{\frac{\epsilon}{2}} + {\frac{\epsilon}{2}}\\
			&=\epsilon,
		\end{split}
	\end{equation*}
	$\forall \hspace{1pt} w \geq N_{0}$, for some $N_{0} \in \NN.$
\end{proof}

% Thus, we have established the boundedness and convergence of Kantorovich-type sampling operators in $L^{\overrightarrow{P}}(\RR^{n})$. In the forthcoming section, we extend the study of these operators within the framework of mixed norm Orlicz spaces.

\section{Kantorovich-type Sampling Operators on mixed norm Orlicz space}
\label{5}

The objective of this section is to establish the convergence of Kantorovich-type sampling operators  in the setting of  mixed norm Orlicz spaces $L^{ \overrightarrow{\Phi}}(\RR^n)$. 
To prove this result, we utilize the boundedness of the operator (\ref{kan}) and the denseness property of $C_c(\RR^n)$ within the framework of $L^{ \overrightarrow{\Phi}}(\RR^n)$.
In the following theorem, we discuss the boundedness of these operators in  $L^{ \overrightarrow{\Phi}}(\RR^n)$  using the modular functional $I^{\overrightarrow{\Phi}}.$
% within mixed norm Orlicz spaces $L^{ \overrightarrow{\Phi}}(\RR^n)$ using the modular functional $I^{\overrightarrow{\Phi}}.$
% As discussed in Section \ref{2}, the concept of modular convergence is well-suitable for Orlicz spaces. In proving the upcoming results of this section, we work within the modular convergence. 
% We begin by establishing the boundedness of Kantorovich-type sampling operator in mixed norm Orlicz spaces.

\begin{thm}\label{thm 5.1}
    Let $f\in L^{\overrightarrow{\Phi}}(\mathbb{R}^n) $  with $\overrightarrow{\Phi}=(\phi_1,\dots,\phi_n)$ and $\lambda >0$, there holds
    \begin{equation} \label{5.1}
    	I^{\overrightarrow{\Phi}}(\lambda K_wf)\leq \frac{\|\Chi\|_1}{\left( m_{0}(\Chi)\right)^n}\hspace{2pt} \hspace{2pt}I^{ \overrightarrow{\Phi}}(\lambda\left(  m_{0}(\Chi)\right)^nf) .
    \end{equation}
\end{thm}

% \color{teal}
% Think about this. Discuss with Shivam, if you want to write it in terms of norm then adjust it accordingly. First, verify these two facts which I have used.
% \begin{itemize}
%     \item $\|\Chi\|_1\leq m_0(\Chi)$,
%     \item $0$-th order moment $m_0(\Chi)\geq 1$.
% \end{itemize}
% \color{black}

\begin{proof}
	First, we examine the result for $n=2$, i.e., for $\overrightarrow{\Phi}=(\phi_1,\phi_2)$.
Assume that the modular term $I^{\overrightarrow{\Phi}}-$ on the right hand side of (\ref{5.1}) is finite. 
 % \begin{flalign*}
 % % I^{(\phi_1,\phi_2)}( \lambda K_wf)&=  \int_{\RR} \phi_{2}\left( \int_{\RR} \phi_{1}\left(\lambda|(K_w f)(x_1,x_2)|\right)dx_1\right)dx_2.
 % \end{flalign*}
From convexity property on the summation and Jensen's inequality on integration, we get 
\begin{flalign*}
	&\int_{\RR}  \phi_1\bigg(\lambda\sum_{k_2 \in \ZZ}\sum_{k_1 \in \ZZ}  \Chi(wx_1-k_1,wx_2-k_2) w^{2} \int_{\frac{k_2}{w}}^{\frac{k_2+1}{w}} \int_{\frac{k_1}{w}}^{\frac{k_1+1}{w}}|f(t_1,t_2)|dt_1dt_2\bigg)dx_1\\
	\leq & \frac{1}{m_{0}(\Chi)}\int_{\RR} \sum_{k_2 \in \ZZ}\sum_{k_1 \in \ZZ} \frac{\Chi(wx_1-k_1,wx_2-k_2)}{m_{0}(\Chi)}\hspace{2pt} \phi_1\left(\lambda m_{0}(\Chi) w^{2} \int_{\frac{k_2}{w}}^{\frac{k_2+1}{w}} \int_{\frac{k_1}{w}}^{\frac{k_1+1}{w}}|f(t_1,t_2)|dt_1dt_2 \right)\\
	\leq & \frac{1}{m_{0}(\Chi)}\sum_{k_2 \in \ZZ}\sum_{k_1 \in \ZZ} \|\Chi(\cdot,wx_2-k_2)\|_1 \int_{\frac{k_2}{w}}^{\frac{k_2+1}{w}}\int_{\frac{k_1}{w}}^{\frac{k_1+1}{w}} w \phi_1 \left(\lambda m_{0}(\Chi)|f(t_1,t_2)|\right)dt_1dt_2\\
		\leq & \frac{1}{m_{0}(\Chi)}\sum_{k_2 \in \ZZ} \|\Chi(\cdot,wx_2-k_2)\|_1 \int_{\frac{k_2}{w}}^{\frac{k_2+1}{w}}\int_{\RR} w \phi_1 \left(\lambda m_{0}(\Chi)|f(t_1,t_2)|\right)dt_1dt_2.
		\end{flalign*}
		Using the convexity and Jensen's inequality on integration, we get
		\begin{align*}
				&I^{\overrightarrow{\Phi}}(\lambda K_wf)\\
			\leq& \frac{1}{m_{0}(\Chi)}\int_{\RR} \phi_2 \left(\lambda \sum_{k_2 \in \ZZ} \|\Chi(\cdot,wx_2-k_2)\|_1 \int_{\frac{k_2}{w}}^{\frac{k_2+1}{w}}\int_{\RR} w \phi_1 \left(\lambda m_{0}(\Chi)|f(t_1,t_2)|\right)dt_1dt_2 \right)dx_2\\
				\leq& \frac{1}{\left(m_{0}(\Chi)\right)^2}\int_{\RR}\sum_{k_2 \in \ZZ} \|\Chi(\cdot,wx_2-k_2)\|_1  \hspace{2pt}\phi_2 \left(\int_{\frac{k_2}{w}}^{\frac{k_2+1}{w}}\int_{\RR} w \phi_1 \left(\lambda\left(m_{0}(\Chi)\right)^2|f(t_1,t_2)|\right)dt_1dt_2 \right)dx_2\\
					\leq& \frac{1}{\left(m_{0}(\Chi)\right)^2}\int_{\RR} \sum_{k_2 \in \ZZ} \|\Chi(\cdot,wx_2-k_2)\|_1 \hspace{2pt}w \int_{\frac{k_2}{w}}^{\frac{k_2+1}{w}}  \phi_2 \left(\int_{\RR}  \phi_1 \left(\lambda\left(m_{0}(\Chi)\right)^2|f(t_1,t_2)|\right)dt_1\right)dt_2\hspace{2pt} dx_2\\
				\leq& \frac{\|\Chi\|_{1}}{\left(m_{0}(\Chi)\right)^2}I^{(\phi_1,\phi_2)}( \lambda\left( m_{0}(\Chi)\right)^2f)\hspace{1pt}.
		\end{align*}
Now	following along the same lines, we obtain
	\begin{equation*}
		\begin{split}
	I^{\overrightarrow{\Phi}}(\lambda K_wf)
	&\leq \frac{\|\Chi\|_1 }{\left( m_{0}(\Chi)\right)^n}\hspace{2pt}\hspace{2pt}I^{ \overrightarrow{\Phi}}(\lambda \left( m_{0}(\Chi)\right)^n f).
	\end{split}
	\end{equation*}
This proves the required estimate.
\end{proof}

\begin{rem}\label{rem 5.2}
The modular inequality on Kantorovich-type sampling operator can be written in terms of norm with respect to mixed norm Orlicz space. Let $f\in L^{\vc{\Phi}}(\RR^n)$ be arbitrary and $\lambda=\|\left(m_0(\Chi)\right)^nf\|_{\vc{\Phi}}$, then by the definition of mixed norm Orlicz space we have $$I^{\vc{\Phi}}\Big(\frac{\left(m_0(\Chi)\right)^nf}{\lambda}\Big)\leq 1.$$ The inequality \eqref{5.1} implies
\begin{align*}
    I^{\vc{\Phi}}\Big(\frac{K_wf}{\lambda}\Big)\leq \frac{\|\Chi\|_1}{\left(m_0(\Chi)\right)^n}I^{\vc{\Phi}}\Big(\frac{\left(m_0(\Chi)\right)^nf}{\lambda}\Big)\leq 1.
\end{align*}
The last inequality hold from the fact that $m_0(\Chi)\geq \max\{1, \|\Chi\|_1\}$.
Therefore, by the definition of mixed  norm Orlicz space, we have
\begin{align*}
    \|K_wf\|_{\vc{\Phi}}\leq \|\left(m_0(\Chi)\right)^nf\|_{\vc{\Phi}}\leq \left(m_0(\Chi)\right)^n\|f\|_{\vc{\Phi}}. 
\end{align*}
\end{rem}

% \dlp{Did you verify the two fact I have mentioned? If so where did you mentioned such result in the paper?}
% \priyanka{Sir, i have verified both facts, but i did not write them clearly in the paper because i thought they were trivial}
% \dlp{This is not trivial to me or seems easily follows from the definition. You need to mention the result, you cannot use it thinking about the reader can understand. Write something like ``It is easy to verify ... from the fact ..."}
% Next, we establish the convergence for \( (K_wf)_{w>0} \) in \( C_c(\mathbb{R}^n) \).
The following theorem establishes the convergence of \( (K_w)_{w>0} \) in \( C_c(\mathbb{R}^n) \)
 considered within the framework of mixed norm Orlicz spaces.

\begin{thm}\label{thm 5.2}
    For $\overrightarrow{\Phi}=(\phi_1,\dots,\phi_n)$, $\lambda>0$ and $f\in C_{c}(\RR^{n})$, we have
    \begin{equation*}  
    	\lim_{w\to\infty}I^{ \overrightarrow{\Phi}}(\lambda (K_wf-f))=0.
    \end{equation*}
\end{thm}
\begin{proof}
	 We begin with $\overrightarrow{\Phi}=(\phi_1,\phi_2)$ for better mathematical visualization. Here we prove that
     \begin{equation*}
     	\lim_{w\to\infty}I^{\overrightarrow{\Phi}}(\lambda(K_{w}f-f))
     	=0,
     	\end{equation*}
    for every $\lambda>0.$ 
We will be using the \emph{Vitali Convergence theorem} for mixed norm space \cite{vitali}. The proof proceeds in the following steps:\par 

\textbf{Step 1.} Utilizing \cite[Theorem 4.1]{multi}, we get $$\lim_{w\to\infty}\phi_2(\phi_1(\lambda\|K_wf-f\|_{\infty}))=0.$$

\textbf{Step 2.} As $f \in C_c(\RR^n)$, i.e., $\text{supp}\, f \subseteq [- a, a]^2$. For every sufficiently large $w>0$,   whenever $(k_1,k_2) \notin [-wa-1,wa]^2$, we have $\left\lbrace [\frac{k_1}{w},\frac{k_1+1}{w}] \times [\frac{k_2}{w},\frac{k_2+1}{w}]\right\rbrace \cap [- a, a]^2=\emptyset$. This provides
  \begin{flalign*} 
  	\label{5.2}
  \int_{\frac{k_2}{w}}^{\frac{k_2+1}{w}}	\int_{\frac{k_1}{w}}^{\frac{k_1+1}{w}}|f(t_1,t_2)|\hspace{2pt}dt_1dt_2=0.
  \end{flalign*}
	By Lemma \ref{cc}, for every fixed $\epsilon >0$, there exists a constant $M>0$ (we can assume $M >  a$) such that 
	\begin{equation*}
		\int_{\|\textbf{x}\|_{\infty}>{M}}w^n|\Chi(w\textbf{x}-\textbf{k})|d\textbf{x} < \epsilon ,
	\end{equation*}
    for every $(k_1,k_2) \in [-wa-1,wa]^2$.
  For a set $\textbf{G}=[-M,M]^2$. We compute on the complement of $\textbf{G}$
 \begin{flalign*}
	J &= \int_{|x_2|>M}\phi_{2}\left(\int_{|x_1|\leq M}\phi_{1}(\lambda|(K_{w}f)(x_1,x_2)|)\,dx_1\right)\,dx_2\\
    & \hspace{1in}+\int_{|x_2|\leq M}\phi_{2}\left(\int_{|x_1|>M}\phi_{1}(\lambda |(K_{w}f)(x_1,x_2)|)\,dx_1\right)\,dx_2 \nonumber \\
	& \hspace{2in}+\int_{|x_2|>M}\phi_{2}\left(\int_{|x_1|>M}\phi_{1}(\lambda|(K_{w}f)(x_1,x_2)|)\,dx_1\right)\,dx_2 \nonumber \\
    &:= J_1+ J_2+J_3. 
\end{flalign*}
    First compute $J_1.$ In view of Jensen's inequality and  Fubini-Tonelli theorem, we obtain 
\begin{flalign*}
  &\int_{|x_1|\leq M}  \phi_1\left(\lambda \sum_{\frac{k_2}{w} \in  [-a-\frac{1}{w}, a]} \sum_{\frac{k_1}{w}\in  [-a-\frac{1}{w}, a]} |\Chi(wx_1-k_1,wx_2-k_2)| w^{2}
\int_{\frac{k_2}{w}}^{\frac{k_2+1}{w}} \int_{\frac{k_1}{w}}^{\frac{k_1+1}{w}}|f(t_1,t_2)|dt_1dt_2\right)dx_1\\
   &\displaystyle \leq  \frac{1}{ m_{0}(\Chi)}    \sum_{\frac{k_2}{w} \in  [-a-\frac{1}{w}, a]} \sum_{\frac{k_1}{w}\in  [-a-\frac{1}{w}, a]} \phi_1\left(\lambda  m_{0}(\Chi)\hspace{2pt}w^2 \int_{\frac{k_2}{w}}^{\frac{k_2+1}{w}}\int_{\frac{k_1}{w}}^{\frac{k_1+1}{w}}|f(t_1,t_2)|dt_1dt_2\right)\\
 &\hspace{3 in} \times\int_{|x_1|\leq M}|\Chi(wx_1-k_1,wx_2-k_2)|dx_1\\  
   &\displaystyle \leq  \frac{1}{m_{0}(\Chi)}  \sum_{\frac{k_2}{w} \in  [-a-\frac{1}{w}, a]} \sum_{\frac{k_1}{w}\in  [-a-\frac{1}{w}, a]}    \hspace{2pt}w^2 \int_{\frac{k_2}{w}}^{\frac{k_2+1}{w}}\int_{\frac{k_1}{w}}^{\frac{k_1+1}{w}}\phi_1\left(\lambda  m_{0}(\Chi)|f(t_1,t_2)|\right)dt_1dt_2\\
     &\hspace{3 in} \times\int_{|x_1|\leq M}  |\Chi(wx_1-k_1,wx_2-k_2)|dx_1\\
       &\hspace{-0.2cm}\displaystyle \leq \frac{1}{m_{0}(\Chi)}  \sum_{\frac{k_2}{w} \in  [-a-\frac{1}{w}, a]}    \hspace{2pt}w \int_{\frac{k_2}{w}}^{\frac{k_2+1}{w}} \int_{-a-\frac{1}{w}}^{a+\frac{1}{w}}\phi_1\left(\lambda  m_{0}(\Chi)|f(t_1,t_2)|\right)dt_1dt_2   \hspace{4pt} \|\Chi(\cdot,wx_2-k_2)\|_{L^{1}([-M,M])}.
    % &\hspace{-0.2cm}\displaystyle < \frac{1}{m_{0}(\Chi)}  \sum_{\frac{k_2}{w} \in [-a,a]}    \hspace{2pt}w \int_{\frac{k_2}{w}}^{\frac{k_2+1}{w}} \int_{-a}^{a+\frac{1}{w}}\phi_1\left(\lambda  m_{0}(\Chi)|f(t_1,t_2)|\right)dt_1dt_2   \hspace{4pt} \|\Chi(\cdot,wx_2-k_2)\|_{L^{1}([-M,M])}.
\end{flalign*}
% Note that the last inequality follows from the Archimedean property, i.e, for any $\frac{1}{w}, w > 0$, we have $\frac{1}{w} < \delta'$.
% \dlp{$\sum_{\frac{k}{w}\in [-a,a]}\int_{\frac{k}{w}}^{\frac{k+1}{w}}$ does not bounded above by $\int_{-c}^c$. Lower limit is correct but not the upper limit.}
% \priyanka{The upper limit is $\frac{c+1}{w}$, but for sufficiently large $w$, it is approximately $\frac{c}{w}$.}
Again using Jensen's inequality and Fubini-Tonelli theorem, we get
\begin{flalign*}
    J_1 
     & \leq \frac{1}{\left(m_{0}(\Chi)\right)^2}   \sum_{\frac{k_2}{w} \in  [-a-\frac{1}{w}, a]}  \phi_{2} \left(   \hspace{2pt}w \int_{\frac{k_2}{w}}^{\frac{k_2+1}{w}}\int_{-a-\frac{1}{w}}^{a+\frac{1}{w}}\phi_1\left(\lambda  \left(m_{0}(\Chi)\right)^2|f(t_1,t_2)|\right) dt_1dt_2  \right) \\
     &\hspace{2.5 in} \times \int_{|x_2|>M} \|\Chi(\cdot,wx_2-k_2)\|_{L^{1}([-M,M])}dx_2\\
       &\leq   \frac{1}{\left(m_{0}(\Chi)\right)^2}   \sum_{\frac{k_2}{w} \in  [-a-\frac{1}{w}, a]}   \hspace{2pt} \int_{\frac{k_2}{w}}^{\frac{k_2+1}{w}}  \phi_{2} \left( \int_{-a-\frac{1}{w}}^{a+\frac{1}{w}}\phi_1\left(\lambda  \left(m_{0}(\Chi)\right)^2|f(t_1,t_2)|\right)dt_1 \right) dt_2\\
      &\hspace{2.5 in} \times \int_{|x_2|>M} w \|\Chi(\cdot,wx_2-k_2)\|_{L^{1}([-M,M])}dx_2\\
        & \hspace{-0.2cm}< \frac{\epsilon}{3 \left(m_{0}(\Chi)\right)^2}   I^{(\phi_1,\phi_2)}\left(\lambda  \left(m_{0}(\Chi)\right)^2 f\right) .
\end{flalign*}
 Note that the last inequality follows from the following observation:
        \begin{equation*}
         \int_{|x_2|>M}  w \|\Chi(\cdot,wx_2-k_2)\|_{L^1([-M,M])} \hspace{2pt} dx_2< \frac{\epsilon}{3}.
        \end{equation*}
% \dlp{How did you justify $\|\Chi\|_{L^1([-M,M]\times [-M,M]^c)}<\epsilon$. You use the Lemma 4.3 which is claims $\|\Chi\|_{L^1([-M,M]^c\times [-M,M]^c)}<\epsilon$.}
Similarly, we can compute $J_2$. Thus we get
$$J_2 <  \frac{\epsilon}{3\left(m_{0}(\Chi)\right)^2}   I^{(\phi_1,\phi_2)}\left(\lambda  \left(m_{0}(\Chi)\right)^2 f\right) .$$
  Let us rewrite $J_3 $ as :
 	\begin{multline*}
  	J_3:\displaystyle  =\int_{|x_2|>M} \phi_{2}\Bigg( \int_{|x_1|>M}  \phi_1\Bigg(\lambda \sum_{\frac{k_2}{w} \in [-a,a]} \sum_{\frac{k_1}{w}\in[-a,a]} |\Chi(wx_1-k_1,wx_2-k_2)| w^{2}\\
    \int_{\frac{k_2}{w}}^{\frac{k_2+1}{w}}\int_{\frac{k_1}{w}}^{\frac{k_1+1}{w}}|f(t_1,t_2)|dt_1dt_2\Bigg)dx_1\Bigg) dx_2.
\end{multline*}
In view of Jensen's inequality and Fubini-Tonelli theorem, we obtain 
\begin{flalign*}
  % \text{Then, by Jensen's inequality and Fubini-Toneli theorem we obtain}\\
  &\int_{|x_1|>M}  \phi_1\left(\lambda \sum_{\frac{k_2}{w} \in  [-a-\frac{1}{w}, a]} \sum_{\frac{k_1}{w}\in  [-a-\frac{1}{w}, a]} |\Chi(wx_1-k_1,wx_2-k_2)| w^{2}
\int_{\frac{k_2}{w}}^{\frac{k_2+1}{w}} \int_{\frac{k_1}{w}}^{\frac{k_1+1}{w}}|f(t_1,t_2)|dt_1dt_2\right)dx_1\\
   &\displaystyle \leq  \frac{1}{ m_{0}(\Chi)}  \sum_{\frac{k_2}{w} \in  [-a-\frac{1}{w}, a]} \sum_{\frac{k_1}{w}\in  [-a-\frac{1}{w}, a]} \phi_1\left(\lambda  m_{0}(\Chi)\hspace{2pt}w^2 \int_{\frac{k_2}{w}}^{\frac{k_2+1}{w}}\int_{\frac{k_1}{w}}^{\frac{k_1+1}{w}}|f(t_1,t_2)|dt_1dt_2\right)\\
    &\hspace{2.5in} \times \int_{|x_1|>M} |\Chi(wx_1-k_1,wx_2-k_2)|dx_1\\  
   &\displaystyle \leq  \frac{1}{m_{0}(\Chi)}  \sum_{\frac{k_2}{w} \in  [-a-\frac{1}{w}, a]} \sum_{\frac{k_1}{w}\in  [-a-\frac{1}{w}, a]}    \hspace{2pt}w^2 \int_{\frac{k_2}{w}}^{\frac{k_2+1}{w}}\int_{\frac{k_1}{w}}^{\frac{k_1+1}{w}}\phi_1\left(\lambda  m_{0}(\Chi)|f(t_1,t_2)|\right)dt_1dt_2\\
   &\hspace{2.5in} \times \int_{|x_1|>M}   |\Chi(wx_1-k_1,wx_2-k_2)|dx_1\\
    &\hspace{-0.2cm}\displaystyle \leq  \frac{1}{m_{0}(\Chi)}  \sum_{\frac{k_2}{w} \in  [-a-\frac{1}{w}, a]}    \hspace{2pt}w \int_{\frac{k_2}{w}}^{\frac{k_2+1}{w}}\int_{-a-\frac{1}{w}}^{a+\frac{1}{w}}\phi_1\left(\lambda  m_{0}(\Chi)|f(t_1,t_2)|\right)dt_1dt_2   \hspace{4pt} \|\Chi(\cdot,wx_2-k_2)\|_{L^1([-M,M]^c)}.
\end{flalign*}
Again using Jensen's inequality and Fubini-Tonelli theorem, we have
\begin{flalign*}
    J_3 
     & \leq  \frac{1}{\left(m_{0}(\Chi)\right)^2}   \sum_{\frac{k_2}{w} \in  [-a-\frac{1}{w}, a]}  \phi_{2} \left(   \hspace{2pt}w \int_{\frac{k_2}{w}}^{\frac{k_2+1}{w}} \int_{-a-\frac{1}{w}}^{a+\frac{1}{w}}\phi_1\left(\lambda  \left(m_{0}(\Chi)\right)^2|f(t_1,t_2)|\right) dt_1dt_2  \right) \hspace{4pt} \\
  &\hspace{2.5in} \times \int_{|x_2|>M} \|\Chi(\cdot,wx_2-k_2)\|_{L^1([-M,M]^c)}dx_2\\
       & \leq  \frac{1}{\left(m_{0}(\Chi)\right)^2}   \sum_{\frac{k_2}{w} \in  [-a-\frac{1}{w}, a]}   \hspace{2pt} \int_{\frac{k_2}{w}}^{\frac{k_2+1}{w}}  \phi_{2} \left( \int_{-a-\frac{1}{w}}^{a+\frac{1}{w}}\phi_1\left(\lambda  \left(m_{0}(\Chi)\right)^2|f(t_1,t_2)|\right)dt_1 \right) dt_2\\
       &\hspace{2.5in} \times  \int_{|x_2|>M} w \|\Chi(\cdot,wx_2-k_2)\|_{L^1([-M,M]^c)}dx_2.
        \end{flalign*}
        Using  Lemmma \ref{cc}, we get
        \begin{flalign*}
         J_3<  \frac{\epsilon}{3\left(m_{0}(\Chi)\right)^2}   I^{(\phi_1,\phi_2)}\left(\lambda  \left(m_{0}(\Chi)\right)^2 f\right)\hspace{4pt}.
\end{flalign*}
Thus
     \begin{flalign*}
        J &<   \frac{\epsilon}{\left(m_{0}(\Chi)\right)^2}   I^{(\phi_1,\phi_2)}\left(\lambda  \left(m_{0}(\Chi)\right)^2 f\right).
     \end{flalign*}

\textbf{Step 3.} For every $\epsilon >0$, $\exists \hspace{3pt} \delta >0$ such that   \begin{equation} \label{leb2}
\left(\int_{B_2} \int_{B_1}|\Chi(u_1,u_2)|du_1du_2 \right) < \frac{\epsilon \left(m_{0}(\Chi)\right)^2}{I^{(\phi_1,\phi_2)}(\lambda\left( m_{0}(\Chi)\right)^2f)}
\end{equation} for every pair of measurable sets $B_i \subset \RR$ with $\displaystyle \mu(B_i)<\delta $, for $i=1,2.$
  Proceeding along the same lines as in the proof of Theorem \ref{thm 5.1} and absolute continuity property of the Lebesgue integral (\ref{leb2}), we have
\begin{flalign*}
 	  \int_{B_2}\phi_{2}\left(\int_{B_1}\phi_{1}(\lambda|(K_{w}f)(x_1,x_2)|)dx_1\right)dx_2 &\leq   \frac{1}{\left(m_{0}(\Chi)\right)^2} I^{(\phi_1,\phi_2)}(\lambda\left( m_{0}(\Chi)\right)^2f)\hspace{1pt} \|\Chi\|_{L^{(1,1)}(B_1 \times B_2)}\\
      & < \epsilon.
 	  % &<  \frac{\epsilon }{\left(m_{0}(\Chi)\right)^2} I^{(\phi_1,\phi_2)}(\lambda\left( m_{0}(\Chi)\right)^2f).
 	\end{flalign*}
  Thus, the integrals 
  $$\int_{B_2}\phi_{2}\left(\int_{B_1}\phi_{1}(\lambda|(K_{w}f)(x_1,x_2)|)dx_1\right)dx_2$$
 are equi-absolutely continuous for arbitrary $B_1$ and $B_2$.
  
Therefore, by utilizing the Vitali convergence theorem along with the arbitrary choice of $\lambda>0$, we obtain $I^{\overrightarrow{\Phi}}(\lambda(K_{w}f-f)) \rightarrow 0$ as $w\to\infty$.
    Similarly, we can extend the result for any $n \geq 3.$
\end{proof}

 It is well established that $C_{c}(\RR)$ is dense in $L^{\phi}(\RR)$ \cite[Theorem 7.6]{mod}. However, to the best of our knowledge, it has not been proven that $C_{c}(\RR^{n})$  is dense in $ L^{\overrightarrow{\Phi}}(\mathbb{R}^n).$ 
To thoroughly explain this idea, we aim to provide detailed information as described below.
 \begin{lemma} \label{ce}
             The space $C_{c}(\RR^{n})$ is dense in $ L^{\overrightarrow{\Phi}}(\mathbb{R}^n).$ 
   \end{lemma}
\begin{proof}
	Our first step is to prove the result for $n=2.$ 
	To establish the desired approximation, we demonstrate that any simple function can be closely approximated by a sequence of compactly supported continuous functions.
	Let $S \subset \RR^2$ be a measurable set. For each $m \in \NN,$ there exist a sequence of compact set $(U_m)$ and sequence of open sets $( V_m)$ such that $U_m \subset S \subset  V_m$ and $\mu( V_m \setminus U_m)<1/m.$ By Urysohn’s lemma, there exists $g_m \in C_c(\RR^2)$ such that
	$$g_m(\textbf{x}) = 
	\begin{cases} 
		1, & \text{for } \textbf{x} \in U_m, \\
		0, & \text{for } \textbf{x} \in \RR^2 \setminus  V_m.
	\end{cases}$$
It follows directly that for all $\textbf{x} \in \RR^2,$  the function $f_m$ satisfies $$\Chi_{U_m}(\textbf{x})\leq  g_m(\textbf{x}) \leq  \Chi_{ V_m}(\textbf{x}).$$  
This implies that $0 \leq |g_m(\textbf{x})-\Chi_{U_m}(\textbf{x})|\leq \Chi_{ V_m \setminus U_m}(\textbf{x})$ for all $\textbf{x} \in \RR^2.$ For every $\gamma >0$, we have
	$$I^{\overrightarrow{{\Phi}}} (\gamma(\Chi_S - g_m ))  
	\leq I^{\overrightarrow{{\Phi}}}(\gamma(\Chi_{ V_m} - \Chi_{U_m} ))  
	\leq I^{\overrightarrow{{\Phi}}}(\gamma( \Chi_{ V_m \setminus U_m} )).$$
Since \( |\Chi_{ V_m\setminus U_m}(\textbf{x})| \leq 1 \) for all \( \textbf{x} \in \mathbb{R}^2 \) and it converges to zero almost everywhere, the dominated convergence theorem for mixed norm (\cite[Section 2]{dct}) implies that  $I^{\overrightarrow{{\Phi}}}(\gamma( \Chi_{ V_m \setminus U_m} )) $ approaches to $zero$  as $m \to \infty.$ 
This establishes that  $I^{\overrightarrow{{\Phi}}} (\gamma(\Chi_S - g_m )) $  approaches to $zero$  as $m \to \infty.$  In other words, for any given $\epsilon>0$ and a simple function $h : \RR^2 \rightarrow \RR$, there exists a compactly supported  continuous function $g:\RR^2 \rightarrow\RR$  such that  
\begin{equation}\label{g1}
\|g-h\|_{\vc{\Phi}}<\frac{\epsilon}{2}.
\end{equation}

% Using density of class of simple functions in $C_c(\RR^2)$, for given $\epsiln>0$ and a simple function $h : \RR^2 \rightarrow \RR$, there exist a compactly supported  continuous function $g:\RR^2 \rightarrow\RR$ and suitable $\gamma>0$ such that $$I^{\overrightarrow{{\Phi}}} (\gamma(g-h))<\frac{\epsilon}{2}.$$ 
    
 It remains to prove that the class of simple functions is dense in $L^{\overrightarrow{\Phi}}(\mathbb{R}^2)$. For $f\in L^{\overrightarrow{\Phi}}(\RR^2),$ there exists an increasing sequence of simple functions $(f_n)_{n \in \NN}$   such that $0\leq f_n \uparrow |f|$ pointwise  as $n \rightarrow \infty$.  Moreover, for any $\alpha>0$, we have $\phi_1(\alpha(|f|-f_n))\leq \phi_1(\alpha f)$ and $\phi_1(\alpha(|f|-f_n))$ approaches to zero as $n \rightarrow \infty.$ Hence, by the \emph{dominated convergence theorem} for mixed norm, we have
 $$ \lim_{n\to \infty} \int _{\RR}\phi_2\left(\int_{\RR}\phi_1\left( \alpha|(|f|-f_n)(x_1,x_2)| \right)dx_1\right)dx_2 = 0.$$ Hence, for any positive function $|f|\in L^{\vc{\Phi}}(\RR^2)$ there exists a non-negative simple function arbitrary close to $|f|$.
 % \dlp{Too many information in one sentence. Confusing to what condition implies other. For example, $(f_n)$ is a sequence of functions converges to $|f|$ pointwise, ``so that" for each $\alpha>0$... certain thing converge to zero. That means sequence $(f_n)$ pointwise convergent to $|f|$ does not implies next convergent. We need to take even a sequence which holds next result. At this I am not sure even that type of sequence exists. Read again carefully.}
\par 
As $f\in L^{\vc{\Phi}}(\RR^2)$, then we can write \( f = f^+ - f^- \), where \( f^+, f^- \in L^{\overrightarrow{\Phi}}(\mathbb{R}^2) \) are non-negative functions defined as follows:
\begin{equation*}
    f^+(x)=\max\{f(x), 0\} \qquad \text{and} \qquad f^-(x)=-\min\{f(x), 0\}.
\end{equation*}
Hence, for any \( \epsilon > 0 \) there exist non-negative simple functions $s^+$ and $s^-$ such that
\[
\|f^+ - s^+\|_{\overrightarrow{\Phi}} < \frac{\epsilon}{4}, \quad \|f^- - s^-\|_{\overrightarrow{\Phi}} < \frac{\epsilon}{4}.
\]
Therefore  for any $f\in L^{\overrightarrow{\Phi}}(\RR^2) $ and $\epsilon>0$, there exists a simple function \( s := s^+ - s^- \) such that
	\begin{equation}\label{g2}
		\|f - s\|_{\overrightarrow{\Phi}} \leq \|f^+ - s^+\|_{\overrightarrow{\Phi}} + \|f^- - s^-\|_{\overrightarrow{\Phi}} < \frac{\epsilon}{2}.
	\end{equation}
	 	 	% In other words, for any $f\in L^{\overrightarrow{\Phi}}(\RR^2) $ and $\epsilon>0$, there exist a simple function $g:\RR^2 \rightarrow \RR$ and suitable $\gamma>0$ such that 
	 	 	% $$I^{\overrightarrow{\Phi}}(\gamma(f-g)) < \frac{\epsilon}{2}.$$
In view of the triangle inequality of mixed norm Orlicz space, and the equations \eqref{g1} and \eqref{g2}, we get $\|f-h\|_{\vc{\Phi}}  <\epsilon.$
% \begin{flalign*}
% \|f-h\|_{\vc{\Phi}} &\leq \|f-g\|_{\vc{\Phi}}+\|g-h\|_{\vc{\Phi}} <\epsilon.
% \end{flalign*}

This establishes the result for $n=2$. Following the similar lines of proof, one can extend this proof for any $n \geq 3.$
% \dlp{Proof looks like hold for every $\lambda>0$.}
\end{proof}
\\

% Now in the subsequent theorem we examine the convergence of $(K_wf)$ in mixed norm Orlicz space $L^{\overrightarrow{\Phi}}(\RR^{n})$.
We now proceed to prove the main theorem of this section, namely the convergence of (\ref{kan}) in mixed norm Orlicz space $L^{\overrightarrow{\Phi}}(\RR^{n})$.

\begin{thm} \label{cg}
	Let $f\in L^{\overrightarrow{\Phi}}(\RR^{n})$ be fixed. Then we have
	$$ \lim_{w\to\infty}\|K_wf-f\|_{ \overrightarrow{\Phi}}=0. $$
\end{thm}
    \begin{proof}
We prove the result using the denseness of $ C_{c}(\RR^{n})$ in $L^{\overrightarrow{\Phi}}(\RR^{n})$ as established in Lemma \ref{ce}. Let $f\in L^{\overrightarrow{\Phi}}(\RR^{n})$ and $\epsilon >0$ be fixed. 
Thus, there exists $g \in C_{c}(\RR^{n})$ such that 
\begin{equation}\label{Ae}
\|f-g\|_{\overrightarrow{\Phi}} <\frac{\epsilon}{2 \left( 1+\left(m_0(\Chi)\right)^n\right) }.
\end{equation}
	% $$ I^{ \overrightarrow{\Phi}}(\gamma(f-g)) < \displaystyle\frac{\epsilon}{2\left(1+ \|\Chi\|_{1}/\left(m_{0}(\Chi)\right)^n\right)}.$$
   %  \dlp{Suppose $f\in L^{p}(\RR^n)$ and $g\in C_c(\RR^n)$ be arbitrary with $\|f\|_p=\|g\|_p=1$. Consider $\gamma=\frac{\epsilon}{2^{p+1}}$, then $I^p(\gamma(f-g))<\epsilon$. So you can choose arbitrary small $\gamma$ to make two functions close. Statement may be true for every $\lambda$, but you restrict $\lambda$ choice.}
   %  \priyanka{Due to the first and third part, i have to choose this particular value of $\lambda$}
   %  \dlp{I didn't mean which part you need those $\gamma$. Concern is choice of $\gamma$. If your definition of denseness depend on $\gamma$, then any arbitrary function can be close to each other. In particular, $10^{10}$ is also close to $1$, i.e., $\gamma(10^{10}-1)<\epsilon$ for $\gamma=10^{-25}$ or even smaller.}
   % \priyanka{The choice of $\gamma$ depends on $f$ and $\epsilon$ we can not choose an arbitrarily small value.} \dlp{Lemma \ref{ce} does not used or mentioned such relation.} \priyanka{But in the case of density, we always proceed this way: we fix one function $f$, and then, based on $f$ and $\epsilon$, we choose $g$. We already used this approach in Theorem 4.5.}
 % Consider $\lambda >0$ such that $3\lambda(1+\left(m_{0}(\Chi)\right)^n) \leq \gamma.$
 Again Theorem \ref{thm 5.2} implies there exists $N_0\in \NN$ such that for each $w\geq N_0$, we have $$\|K_wg-g\|_{\vc{\Phi}}<\frac{\epsilon}{2}.$$
Now for $w\geq N_0$, the triangle inequality of $ \|\cdot\|_{\overrightarrow{\Phi}}$ along with \eqref{Ae} and Remark \ref{rem 5.2}, we obtain
	\begin{align*}
			\| K_wf-f\|_{\overrightarrow{\Phi}} &\leq \|g-f\|_{  \overrightarrow{\Phi}}+ \|K_wf-K_wg\|_{\vc{\Phi}}+ \|K_wg-g\|_{\vc{\Phi}}\\
			&<  \left( 1+\left(m_0(\Chi)\right)^n\right) \|g-f\|_{  \overrightarrow{\Phi}} + \frac{\epsilon}{2}\\
			&<\epsilon.
	\end{align*}
	This completes the proof.
\end{proof}

% With this we have established the boundedness and convergence of a family of sampling operators in the proposed mixed norm spaces.
\section{Examples of kernel} \label{6}
 % In our theoretical framework kernels hold significant importance.
 Kernels hold significant importance in the theoretical framework.
In this section, we present several examples of suitable kernels that satisfy assumptions (\ref{re1}) and (\ref{re2}) of our presented theory. Assume that for each $i=1,\dots,n$, $\Chi_i:\RR\to \RR$ is a kernel satisfies
  $$\displaystyle \sum_{{k_i}\in \ZZ}\Chi_{i}({u_i}-{k_i})=1, \qquad u_i\in \RR,$$
  and 
  $$  m_{0}(\Chi_i) := \sup_{u_i \in \RR}\sum_{k_i \in \ZZ}|\Chi_i(u_i-k_i)|<+\infty.$$
Then the kernel $\Chi:\RR^n\to \RR$ is defined by
$$\displaystyle \Chi(\textbf{u})=\prod_{i=1}^{n}\Chi_{i}(u_i),$$ 
belongs to $L^{1}(\RR^n),$ as
$$\displaystyle \int_{\RR^n}|\Chi(\textbf{u})|d\textbf{u}=\int_{\RR^n}\prod_{i=1}^{n}|\Chi_{i}(u_i)|du_1...du_n=\prod_{i=1}^{n}\int_{\RR}|\Chi_{i}(u_i)|du_i < \infty.$$ 
Moreover, from the above observations, we have
$$ \sum_{\textbf{k}\in \ZZ^n}\Chi(\textbf{u}-\textbf{k})=\prod_{i=1}^{n}\sum_{k_i\in \ZZ}\Chi_i(u_i-k_i)=1,$$ and
$$m_{0}(\Chi)= \sup_{\textbf{u} \in \RR^n}\sum_{\textbf{k}\in \ZZ^n}|\Chi(\textbf{u}-\textbf{k})|=\prod_{i=1}^{n}m_{0}(\Chi_i)<+\infty.$$

% We outline some equivalent conditions related to (\ref{re1}) and (\ref{re2}).
\begin{rem} (\cite{bardaro10})\label{rem}
	(a)  If $\Chi(x)=\mathcal{O}(x^{-1-\epsilon-\alpha})$  for $x \rightarrow \pm \infty$ and some $\epsilon>0,$ the kernel $\Chi$ satisfies (\ref{re2}). \\
	(b) The condition \eqref{re1} is equivalent to the following condition
	$$
	\widehat{\Chi}(k)=
	\begin{cases}
		{1}, &\quad\text{} \ \  {k=0}\\
		{ 0,} &\quad\text{} \ \  {k \neq 0},\\
	\end{cases}$$\\
	where $\widehat{\Chi}(u) := \int_{\RR} \Chi(v)e^{-ivu}\,dv,$ $u \in \RR$ is the Fourier transform of $\Chi$.
\end{rem}

Next, we provide several important examples of kernels that satisfy the assumptions of our proposed theory.\\

\textbf{(I) Fej\'er kernel:}
% \paragraph{Fej\'er kernel:}
% \subsection{Fej\'er kernel}
The Fej\'er kernel is defined as (\cite{bardaro10,multi})
$$F(x) :=\frac{1}{2}\,\sinc^{2}\left(\frac{x}{2}\right),\qquad x \in \RR.$$
% where
% $$
% \sinc(x) :=
% \begin{cases}
% 	{\dfrac{\sin \pi x}{\pi x},} &\quad\text{} \ \  { x \neq 0,}\\
% 	{ 1,} &\quad\text{} \ \  {x=0}.\\
% \end{cases}
% $$\quad\quad\quad\\
It can be observed that the Fourier transform of sinc-function is given by
$$
\widehat{\sinc}(v) =
\begin{cases}
	{1} &\quad\text{} \ \  {|v| \leq \pi},\\
	{ 0,} &\quad\text{} \ \  {|v|>\pi}.\\
\end{cases}
$$
 Indeed $F \in L^1(\RR)$ is bounded. In view of Remark \ref{rem}(a), one can verify that $F$ satisfies \eqref{re2}. Furthermore, the Fourier transform of $F$ is given by
$$
\widehat{F}(v)=
\begin{cases}
	{1-|\frac{v}{\pi}|}, &\quad\text{} \ \  {v \leq \pi},\\
	{ 0,} &\quad\text{} \ \  {v>\pi}.\\
\end{cases}$$\\
% Based on Remark \ref{rem}(b), one can say that F satisfies assumption (\ref{re1}).
Let $\displaystyle \mathcal{F}(\textbf{x}):=\prod_{i=1}^{n}  F(x_i)$ be the multivariate Fej$\mathbf{\acute{e}}$r's kernel. Then $\mathcal{F}$ fulfills the conditions (\ref{re1}) and (\ref{re2}) associated with the generalized kernel (see \cite{multi}).\\

In order to reduce truncation error, one can take \(\Chi(x)\) with faster decay as \(x \rightarrow \pm \infty\). One prominent example in this direction is the Jackson kernel.

\textbf{(II)  Jackson kernel:}
% \paragraph{ Jackson kernel:}
% \subsection{ Jackson kernel}
The Jackson kernel with parameters  $\beta \geq 1,m \in \mathbb{N}$ is given by (\cite{bardaro10})
$$ \overline{J_{\beta,m}}(x):= C_{\beta,m}\ \sinc^{2m} \left(\frac{ x}{2 \beta m \pi} \right) \quad (x \in \RR),$$
with $\displaystyle C_{\beta,m} := \left[\int_{\RR} \sinc^{2m} \left(\frac{x}{2 \beta m \pi} \right){dx}\right]^{-1}$.
\par
It can be observed that the Fourier transform of the Jackson kernel vanishes outside the interval, i.e., $\overline{J_{\beta,m}}(x)$ are bandlimited to $[-{1}/{\beta},{1}/{\beta}]$.  As a consequence of Remark \ref{rem}(a) and (b), we conclude that (\ref{re1}) and (\ref{re2}) are satisfied. 
\par
 As in the previous case, the multivariate Jackson kernel given by 
 $\displaystyle \mathcal{J}(\textbf{x}) = \prod_{i=1}^{n} \overline{J_{\beta,m}}(x_i),$
 satisfies all the desired properties.\\

One can minimize the truncation error by selecting a suitable kernel $\Chi$ with compact support. A noteworthy example in this direction is the well-known B-spline kernel.

\textbf{(III) B-spline kernel: }
% \paragraph{B-spline kernel:}
% \subsection {B-spline kernel}
The B-spline kernel of order $m \in\NN$ is defined as (\cite{bardaro10,multi})
$${B}_{m}(t):= \frac{1}{(m-1)!} \sum_{j=0}^{m} (-1)^{j} {m \choose j} \bigg( \frac{m}{2}+ t-j \bigg)_{+}^{m-1},\qquad t\in \mathbb{R},$$ where $(x)_{+} := \max \{x,0\},$ $ x \in \mathbb{R}.$ 
The Fourier transform of $B_m$ is given by
\begin{equation*} \label{e4}
	\widehat{B_{m}} (v)=
	\sinc^m\left(\frac{v}{2\pi}\right).
\end{equation*}
% By using Poisson summation formula (\cite{past}) the following holds
% \begin{eqnarray} \label{summation}
% 	\displaystyle\sum_{k=- \infty}^{\infty}{B}_{m}(x)=\sum_{k=- \infty}^{\infty}\widehat{B_{m}} (2k\pi )e^{2k\pi ix}. 
% \end{eqnarray}
% From (\ref{summation}), we obtain
% $$\displaystyle\sum_{k=- \infty}^{\infty}{B}_{m}(x)=1, \ \ \forall x\in\mathbb{R}^{+}.$$ 
Based on Remark \ref{rem}(b) we conclude that $ {B}_{m}$ satisfies $(\ref{re1})$. Furthermore, ${B}_{m}$ is bounded on $\RR$ for all $m \in \NN $, with compact support contained in $[-m/2,m/2].$ Therefore, we can say that $B_m \in L^1(\RR)$.  Moreover, the moment condition is satisfied for all $\alpha>0$.
As in the previous case, the multivariate B-spline kernel can be written as
$$\displaystyle \mathcal{B}(\textbf{x})=\prod_{i=1}^{n}  B_m(x_i),$$ which satisfies all the properties of a kernel.\\ 

Another example in this direction is the  Bochner–Riesz kernel.

\textbf{(IV) Bochner-Riesz kernel:}
% \paragraph{Bochner-Riesz kernel:}
% \subsection{Bochner-Riesz kernel}
The Bochner–Riesz kernel is defined as (\cite{nessel})

\[
 B^{\gamma}(\textbf{x}) = \frac{2^{\gamma}}{\left(\sqrt{2\pi}\right)^{N}}  \hspace{2pt}\Gamma(\gamma + 1)\left(|\textbf{x}|\right)^{-(N/2 - \gamma)} J_{\left(N/2 \right) + \gamma}(|\textbf{x}|), \quad  \textbf{x}\in \RR^n
\]
for $\gamma > 0$, where $J_{\lambda}$ is the Bessel function of order $\lambda$ and $\Gamma$ is Gamma function. The Fourier transform of $ B^{\gamma}$ is given by
\[
\widehat{B}^{\gamma}(\boldsymbol{\xi}) =  
\begin{cases} 
(1 - |\boldsymbol{\xi}|^2)^{\gamma}, & |\boldsymbol{\xi}| \leq 1, \\
0, & |\boldsymbol{\xi}| > 1.
\end{cases}
\]
In view of Remark \ref{rem}, we conclude that $B^{\gamma}$ satisfies $(\ref{re1})$ and $(\ref{re2})$.

% it follows that $$\displaystyle\sum_{\textbf{k} \in \ZZ^n}B^{\gamma}(\textbf{u}-\textbf{k})=\widehat{B}^{\gamma}(0)=1.$$ Hence, we conclude that  $B^{\gamma}$ satisfies $(\ref{re1})$ and $(\ref{re2})$.

\section{Conclusion}

This article aims to study the approximation properties of a well-known family of sampling operators within a broad framework of mixed norm function spaces. This not only contributes to approximation theory, but also explores mixed norm spaces in function theory. To analyze the boundedness of generalized sampling operators in mixed norm Lebesgue spaces, we have introduced an appropriate subspace of \( L^{\overrightarrow{P}}(\mathbb{R}^n) \), denoted by \( \Delta^{\overrightarrow{P}}(\mathbb{R}^n) \), in Section~\ref{3}. 
Further, we have established the boundedness and convergence of Kantorovich-type sampling operators within the broad framework of mixed norm Lebesgue spaces and mixed norm Orlicz spaces. To prove the proposed results, we have proved a fundamental density result stating that the space of compactly supported continuous functions $C_c(\RR^n)$ is dense in mixed norm Orlicz space $ L^{\overrightarrow{\Phi}}(\mathbb{R}^n).$ 

\section{Information.} 

\subsection{Ethic.}
\noindent 
The authors declare that the present work is original and has not been published elsewhere nor is it under consideration for publication elsewhere. No part of the manuscript has been copied or plagiarized from other sources, and all results and references have been properly acknowledged.

\subsection{Funding.} 
 \noindent The first author is supported by the University Grants Commission (UGC), New Delhi, India, under NTA Reference No : 231610034215. The third author acknowledges funding from the Deutsche Forschungsgemeinschaft (DFG, German Research Foundation) - Project number 442047500 through the Collaborative Research Center ``Sparsity and Singular Structures” (SFB 1481).

\subsection{Acknowledgement}
\noindent 
Priyanka Majethiya and Shivam Bajpeyi express their gratitude to SVNIT Surat, India, for providing the necessary facilities to carry out this research work.

\subsection{Data availability.} 
\noindent 
Data sharing is not applicable to this article, as no data sets were generated
or analyzed during the current study.

\subsection{Authorship contribution.}
\noindent
\textbf{Priyanka Majethiya}: Writing – Original Draft, Writing – Review and Editing, Conceptualization, Formal Analysis, Methodology, Mathematical Proofs, Visualization. \\
\textbf{Shivam Bajpeyi:} Writing – Original Draft, Writing – Review and Editing, Conceptualization, Formal Analysis, Methodology, Mathematical Proofs, Visualization, Supervision.\\
\textbf{Dhiraj Patel:} Writing – Original Draft, Writing – Review and Editing, Conceptualization, Formal Analysis, Methodology, Mathematical Proofs, Visualization, Supervision.

\subsection{Conflict of interest.} 
\noindent 
The author declares no conflict of interest related to the content of this article.

\end{document}